\documentclass[11pt]{article}
\usepackage[a4paper, margin=1in]{geometry}
\usepackage{graphicx,times,amsmath,amsthm,amsfonts,amssymb,mathtools,color,latexsym,amsxtra,layout,bibentry,hyperref}
\usepackage{subfig}
\usepackage[dvipsnames]{xcolor}
\usepackage{graphicx}
\usepackage[shortlabels]{enumitem}

\usepackage{array}
\newcolumntype{L}[1]{>{\raggedright\let\newline\\\arraybackslash\hspace{0pt
}}m{#1}}
\newcolumntype{C}[1]{>{\centering\let\newline\\\arraybackslash\hspace{0pt}}
m{#1}}
\newcolumntype{R}[1]{>{\raggedleft\let\newline\\\arraybackslash\hspace{0pt}
}m{#1}}

\usepackage{times,amsmath,stackengine,amsthm,amsfonts,amssymb,mathtools,color,latexsym,amsxtra,layout}
\usepackage{graphicx}
\usepackage{algorithmicx, algpseudocode}
\usepackage{algorithm}
\usepackage{multirow}
\usepackage{comment, centernot}

\newcommand{\inner}[1]{\left \langle #1  \right \rangle}

\newcommand{\ltwoinner}[1]{\left \langle #1  \right \rangle_{2}}
\newcommand{\ltwof}[1]{\left \langle #1  \right \rangle_{f}}
\newcommand{\ltwoD}[1]{\left \langle #1  \right \rangle_{D}}
\newcommand{\ltwoDf}[1]{\left \langle #1  \right \rangle_{D,f}}

\newcommand{\norm}[1]{\left \| #1 \right \|_2}
\newcommand{\fnorm}[1]{\left \| #1 \right \|_f}
\newcommand{\Dnorm}[1]{\left \| #1 \right \|_D}
\newcommand{\Dfnorm}[1]{\left \| #1 \right \|_{D,f}}
\newcommand{\Mnorm}[1]{\left \| #1 \right \|_M}

\newcommand{\ltwonorm}[1]{\left \| #1 \right \|_{2}}

\newcommand{\linftynorm}[1]{\left \| #1 \right \|_{\infty}}

\newcommand{\hnorm}[1]{\left \| #1 \right \|_{h}}

\usepackage{graphicx} 
\theoremstyle{plain}
\newtheorem{theorem}{Theorem}[section]
\newtheorem{lemma}[theorem]{Lemma}
\newtheorem{remark}[theorem]{Remark}
\newtheorem{o-thm}[theorem]{Theorem}

\newtheorem{corollary}[theorem]{Corollary}

\theoremstyle{definition}

\title{Direct Problem for Gas Diffusion in Polar Firn\\ with\\ Variable Coefficients}
\author{Sophie Moufawad \thanks{American University of Beirut (AUB), Beirut, Lebanon.  (sm101@aub.edu.lb) } \and Nabil Nassif \thanks{American University of Beirut (AUB), Beirut, Lebanon.  (nn12@aub.edu.lb) }\;\,\thanks{This work was supported by the AUB University Research Board grant number 104261 (Project 26742).}  \and Faouzi Triki  \thanks{Laboratoire Jean Kuntzmann, UMR CNRS 5224, Université Grenoble-Alpes, 700 Avenue Centrale, 38401 Saint-Martin-d’Hères, France  (faouzi.triki@univ-grenoble-alpes.fr) }}

\begin{document}

\maketitle
\begin{abstract}
We consider the mathematical model of gas trapping in deep polar ice (firns), which consists
of a parabolic partial diﬀerential equation, that can degenerate at one
boundary extreme. 
In \cite{FirnA}, we considered all the coefficients to be constants, except the diﬀusion coeﬃcient $D(z)$ that is to be reconstructed. In this paper, we assume both the diﬀusion coeﬃcient $D(z)$ and the volume fraction $f(z)$ are functions. 
The difficulty in this problem, both theoretically and computationally, arises from the fact that $D(z)$ and $f(z)$ may be zero at bottom of the firn.  
To handle such degeneracy, we defined appropriate weighted Sobolev spaces and used Lion's theorem to prove existence and uniqueness of the semi-variational formulation of the Firn PDE. A full discrete system is obtained through a P1 Finite element Galerkin procedure in space and an Euler-Implicit  scheme in time. Sufficient conditions for the existence and uniqueness of the solution for the discrete system are obtained.
\end{abstract}

\section{Introduction}
The polar ice and snow of Antarctica and Greenland serve as an exceptional archive of historical climate and atmospheric conditions. Owing to a well-established understanding of the mechanisms governing gas entrapment in deep ice; particularly the processes of densification and pore closure within firn layers, typically occurring in the upper hundred meters of the ice sheet; several models have been developed to describe these phenomena.

We consider the model developed in \cite{Witrant2012acp}, where the concentration $ \rho_\alpha ^{\rm o} $ of a gas $ \alpha$ in open pores satisfies an initial-value, time-dependent advection-diffusion partial differential equation on a one-space dimension segment $[0,z_F]$  with Dirichlet boundary condition at $0$ and a mixed one at $z_F$, where $\rho_\alpha^{\text{atm}}(0)=0$. For  $\; z\in (0, z_{\rm F}),  t>0$:
\begin{equation}\label{eq:trace_gas_dynamics}
\left\{
\begin{array}{l}
\displaystyle{\frac{\partial}{\partial t} [\rho_\alpha^{\rm o} f] + \frac{\partial}{\partial z} [\rho_\alpha^{\rm o} f ({v}+{w}_{\rm air})] + \rho_\alpha^{\rm o} (\tau+\lambda) =\frac{\partial}{\partial z} \left[D_\alpha \left(\frac{\partial \rho_{\rm \alpha}^{\rm o}}{\partial z} - {\rho}_{\alpha}^{\rm o} \frac{M_{\alpha}g}{RT} \right)\right]}, \vspace{1mm}\\
\rho_\alpha^{\rm o}(0,t) = \rho_\alpha^{\rm atm}(t), \; t>0,\vspace{2mm} \\
\displaystyle{ D_\alpha(z_F)\left(\frac{\partial {\rho}_{\alpha}^{\rm o} }{\partial z}(z_{\rm F},t) -  \frac{M_{\alpha}g}{RT} {\rho}_{\alpha}^{\rm o}(z_{\rm F},t)\right) = 0},\\
\rho_\alpha^{\rm o}(z,0) = 0
\end{array}
\right.
\end{equation}
Moreover, $ f(z) $ is the average volume fraction in the open pores and $ D_\alpha(z) $ is the effective diffusion coefficient of the gas $\alpha$ in the Firn ($m^2/yr$) which is given by 
\begin{eqnarray} \label{TT}
D_\alpha(z) = \left\{ \begin{array}{llcc}
D_{\rm{eddy}}(z)+r_\alpha c_f D_{\textrm{CO2, air}}(z)\;\; \textrm{ if } z \leq z_{\textrm{eddy}},\\
r_\alpha D_{\textrm{CO2, air}}(z),\;\; \textrm{ if } z > z_{\textrm{eddy}},
\end{array}
\right.
\end{eqnarray}
with $ z_{\textrm{eddy}}$, $ r_\alpha $, and $ c_f $ are known constants, and $D_{\rm{eddy}}(z), D_{\textrm{CO2, air}}(z)$ diffusion coefficients. 
The remaining terms are considered constants in this paper,  and
 summarized in Table \ref{tab:results}.
\begin{table}[H]\label{tab}
\setlength{\tabcolsep}{3pt}
\caption{The description of the model's parameters. }\label{tab:results}
  \centering
\begin{tabular}{|l|l|}\hline
$ z_{\rm F} $ & the depth of the Firn (m) \\ \hline
 $ v $ & the average descending speed in the Firn (m/yr)\\\hline
  $ w_{air} $ & the average speed of the air (m/yr) 
  \\\hline
  $ \tau $ & the  mass exchange rate between open and closed pores ($/yr$) \\\hline
 $ \lambda $ & the rate of radioactive decay ($/yr$) $\in [0.5, 0.999] $\\\hline
 $ M_\alpha $ & the molar mass of the gas ($kg/mol$) $\in [0.004, 0.133]$; \;\;\; $M_{CO_2} \approx 0.044$ \\\hline
 $ g $ & the gravitational acceleration ($m/s^2$) $\approx 9.80665$\\ \hline
 $ R $ & the universal constant of ideal gases ($J/mol/K$) = $8.314 $ \\\hline
  $ T $ & the mean temperature of the Firn ($K$) $\approx -31+273.15 \approx 242 K$ \\\hline
  $ \rho_\alpha^{\rm atm} $ & the concentration of gas in the atmosphere ($mol/m^3$ of void space) \\\hline
\end{tabular}
\end{table}
\noindent In \cite{Moufawad2024Firn}, we studied the direct and inverse problem (for reconstructing the diﬀusion coeﬃcient) considering $f(z)$ to be constant.  We proved the existence and uniqueness of a solution for such direct problem and build a robust simulation algorithm. An algorithm
for computing the gradient of the objective function of the inverse problem was proposed and its eﬃciency tested using diﬀerent minimization techniques available in MATLAB’s optimization toolbox. \\

\noindent In this paper, we consider the same model where $f(z)\geq 0$ and $D(z)\geq 0$ are positive functions that may degenerate at $z_F$. Our ultimate goal is to study the inverse problem where the diffusion coefficient $D(z) $ and/or the volume fraction $f(z)$ are recovered given the FIRN data. For that purpose, we need to first study and discretize the Direct problem.\\

\noindent In what follows we denote $D := D_\alpha$, ${ \rho^\text{atm}(t) = \rho_\alpha^\text{atm}(t) }$,  $\rho:=\rho^0_\alpha$, and let $\mathcal{M} = \dfrac{M_\alpha g}{RT_m} > 0$, $\mathcal{G} = \tau + \lambda > 0$ and $\mathcal{F} = v + w_\text{air} > 0$.  Then \eqref{eq:trace_gas_dynamics}
 becomes:
\begin{equation} \label{eq:Witrant}
\begin{cases}
{\dfrac{\partial}{\partial t}[f\rho ] + \dfrac{\partial}{\partial z}[f\mathcal{F}\rho ] + \mathcal{G} \rho  = \dfrac{\partial}{\partial z}\left[D\dfrac{\partial\rho}{\partial z} -  D\mathcal{M}\rho\right]},\qquad \forall t\in (0,T] \\
{\rho(0,t) = \rho^\text{atm}(t)}, \quad t > 0, \\
{D(z_F)\left(\dfrac{\partial\rho}{\partial z}(z_F,t) - \mathcal{M} \rho(z_F,t)\right) = 0}, \\
{\rho(z,0) = 0}.
\end{cases}
\end{equation}
\noindent We assume in this work that:
\begin{eqnarray}
 &\bullet& f\in C(0,z_F) \mbox{ with }0  \leq f(z) \leq f_{max},\; \mbox{ and } f(z)>0 \;\;\forall z \in [0,z_F). \label{A10} \\
&\bullet&D\in C(0,z_F) \mbox{ with }0 \leq  D(z) \leq D_{max},   \mbox{ and } D(z)>0 \;\;\forall z \in [0,z_F).\qquad\qquad\qquad
\label{A20} \\
&\bullet & |D(z)-D(y)|<L |z-y|,\quad \forall z,y \in [0,z_F] ,  \text{ where } L >0. \quad \mbox{(Lipschitz Continuity)} \qquad \label{A2pp0}\\
&\bullet & \rho^{\text{atm}}(t) \in H^{1}(0,T)\label{eq:rhoatm}\vspace{4mm}\\
&\bullet & \sup_{z\in (0,z_F)} \dfrac{D(z)}{f(z)}<\infty \mbox{ and }  \sup_{z\in (0,z_F)} \dfrac{f(z)}{D(z)}<\infty \hspace{80mm}\label{A2p2}
\end{eqnarray}
Note that condition \eqref{A2p2} implies that the degeneracies of both $f(z)$ and $D(z)$ are similar.

\noindent We first derive the semi-variational formulation and prove the existence and uniqueness of its solution (section \ref{sec:semiv}). Then, in section \ref{sec:discch} we discretize the rescaled semi-variational formulation using Euler-Implicit time scheme and P1 Finite element space scheme, put it in matrix form, and study the properties of the obtained matrices leading to the proof of uniqueness of the discrete solution.
\section{Semi-Variational Formulation}\label{sec:semiv}
Let $\phi \in H^1(0,z_F)$ such that $\;\phi(0)=0$. Multiplying \eqref{eq:Witrant} by $\phi$,  integrating by parts with respect to $z$ from zero to $z_F$, and using the boundary condition $D(z_F) \bigl[\rho_z(z_F,t) - \mathcal{M} \rho(z_F,t)\bigr] = 0$, we get: 
\begin{eqnarray}
&&\int_0^{z_F} f \rho_t \phi \hspace{0.1cm} dz+ \int_0^{z_F} (f \rho \mathcal{F})_z \phi  \hspace{0.1cm} dz+ \int_0^{z_F} \rho \mathcal{G} \phi  \hspace{0.1cm} dz = \int_0^{z_F} [D (\rho_z - \rho \mathcal{M})]_z \phi dz\nonumber\\
&&\int_0^{z_F} \rho_t f \phi \, dz +\mathcal{F} f(z_F) \rho(z_F,t) \phi(z_F)- \mathcal{F} \int_0^{z_F} \rho f {\phi_z} \, dz + \mathcal{G} \int_0^{z_F} \rho \phi \, dz = -\int_0^{z_F} D (\rho_z - \rho \mathcal{M}) {\phi_z} \, dz
\nonumber
\end{eqnarray}
Therefore one gets
\begin{equation}\label{eq2.8}
\langle f\rho_t , \phi \rangle  + \mathcal{A}(\rho,\phi)=0
\end{equation}
where the bilinear form:
\begin{equation}\label{eq::bilin}
    \mathcal{A}(\rho,\phi)=\langle D \rho_z, {\phi_z} \rangle -\langle  (\mathcal{M} D+f\mathcal{F})\rho, {\phi_z} \rangle + \mathcal{G} \langle \rho, \phi \rangle + \mathcal{F} {f(z_F)} \rho(z_F,t) \phi(z_F)
\end{equation}
\normalfont \noindent So, the variational Firn problem is to find $\rho : [0,T] \times [0, z_F] \to \mathbb{R}$ such that $\forall t > 0$:

\begin{equation}\label{eq::system}
\begin{cases}
 \rho(\cdot,t) \in U_{ad}=\mathcal{T} + \{\rho^ {\text{atm}}\} &\\
\langle \rho_t, \phi \rangle_f + \mathcal{A}(\rho,\phi) = 0,&\,\forall \phi\in \mathcal{T} \\
\rho(z,0) = 0.
\end{cases}
\end{equation}
To prove the existence and uniqueness of a solution to the variational formulation, we use Lion's theorem \ref{th:lions}.
\begin{theorem}[Lion's theorem]\label{th:lions} Let $V$ and $H$ be  Hilbert spaces satisfying:
\begin{equation}
V \subset H \subset V^* \text{ (the dual of } V)
\end{equation}
\noindent with the injection from $V$ to $H$ is dense and continuous.
\noindent Assuming a bilinear form $a(\cdot,\cdot): V \times V \to \mathbb{R}$ satisfies
\begin{equation}
\begin{cases}
|a(v,w)| \leq M \|v\|_V \|w\|_V \\
|a(v,v)| \geq c \|v\|_V^2 - c_1 \|v\|_H^2
\end{cases}
\label{eq25}
\end{equation}
\noindent then for $u_0 \in H$ and $F(t) \in L^2(0,T;V^*)$, the initial value problem
\begin{equation}
\begin{cases}
\inner{u_t,v}_H + a(u(t),v) = \inner{ F(t), v }_{H} \\
u(0) = u_0
\end{cases}
\end{equation}
\noindent admits a unique solution u, satisfying:
\begin{equation}
u \in L^2(0,T;V) \cap C(0,T;H), \quad \frac{du}{dt} \in L^2(0,T;H).
\end{equation}
\end{theorem}
\noindent To apply Lions theorem, we first reformulate \eqref{eq::system} by considering the translated concentration:
$$\overline{\rho}:=\rho - \rho^ {\text{atm}}.$$
It is easily verified that $\overline{\rho}$ must solve the following system:\\
 Find
$\overline{\rho} : [0,T] \times [0, z_F] \to \mathbb{R}$ such that:
\begin{equation}
\left\{
\begin{array}{l}
 \forall t > 0,\, \overline{\rho}(\cdot,t) \in \mathcal{T} \\
\langle \overline{\rho}_t, \phi \rangle_f + \mathcal{A}(\overline{\rho},\phi) = -\langle \rho^ {\text{atm}}(t))',\phi\rangle-\mathcal{A}(\rho^ {\text{atm}}(t),\phi),\,\forall \phi\in \mathcal{T} \\
\overline{\rho}(z,0) = -\rho^ {\text{atm}}(0) = 0.
\end{array}
\right.
\label{eq:system1}
\end{equation}
 We start by defining appropriately the spaces $V, H, \mathcal{T}$ in section \ref{sec:gFunc} and then proceed to proving the existence and uniqueness of the solution to system \eqref{eq:system1} in Section \ref{sec:exist}. 

\subsection{Function Spaces}\label{sec:gFunc}
We define the weighted function space,
\[
L_f^2 (0,z_F) 
= \left\{ v: (0,z_F) \rightarrow \mathbb{R}  \;\Big|\; \fnorm{v}^2 =  \int_0^{z_F} f(z) v^2(z) \, dz < \infty \right\}
\]
where the integral \eqref{eq:innp} defines an inner product as proven in Appendix \ref{sec:app}, denoted by $L^2_f$ inner product
\begin{equation}\label{eq:innp}
\langle u, v \rangle_f := \int_0^{z_F} f(z) u(z)v(z) \, dz,
\end{equation}
which induces the $L^2_f$ norm $\|v\|_f = \sqrt{\langle v,v \rangle_f}$. Note that $L^2_f$ is a Hilbert space assuming $ f(z)>0$ for $z\in[0,z_F)$ (assumption \eqref{A10}) with $L^2(0,z_F) \subset L^2_f(0,z_F)$ since
\begin{equation}\label{eq:uppBound}
    \fnorm{v}^2 \leq f_{max}\ltwonorm{v}^2 =\linftynorm{f}\ltwonorm{v}^2.  
\end{equation}
\noindent Similarly, we define the $L^2_D$ space having the weight $D(z)$ by
\[
L^2_D(0,z_F) = 
\left\{v: (0,z_F) \rightarrow \mathbb{R} \;\Big|\; \|v\|_D^2 := \int_0^{z_F} D(z) v(z)^2 \, dz < \infty \right\} ,
\]
and the integral \eqref{eq:innD} defines an $L^2_D$ inner product (Proof identical to that of $L^2_f $ in Appendix \ref{sec:app})\vspace{-1mm}
\begin{equation}\label{eq:innD}
\langle u, v \rangle_D := \int_0^{z_F} D(z) u(z) v(z) \, dz\vspace{-1mm}
\end{equation}
with the induced norm $\Dnorm{v} = \sqrt{\ltwoD{v,v}}$.  $L_D^2(0,z_F)$ is a Hilbert space with respect to the $L^2_D$ norm assuming \eqref{A20}, with $L^2(0,z_F) \subset L^2_D(0,z_F)$.\\

 \noindent Also, we define the weighted $H^1$ Sobolev space.\vspace{-1mm}
\begin{eqnarray}
H^1_{D,f}(0,z_F)
&=& \left\{ v \in L^2_f(0,z_F)
\;\middle|\;
v_z-weak \in L^2_D(0,z_F) \right\},\nonumber
\end{eqnarray}
where the integral \eqref{eq:dfinn} defines the $H^1_{D,f}$ inner product (Proof in Appendix \ref{sec:appDf})\vspace{-1mm}
\begin{equation}\label{eq:dfinn}
\ltwoDf{ u, v }
:= \ltwof{ u, v } + \ltwoD{ u_z, v_z }
= \int_0^{z_F} \left( f(z)u(z)v(z) + D(z)u_z(z)v_z(z) \right)\, dz,\vspace{-1mm}
\end{equation}\vspace{-1mm}
\noindent with  the associated induced norm
\[
\Dfnorm{ v} = \left(\ltwoDf{ v, v }\right)^{\frac{1}{2}} =\left( \ltwof{ v, v } + \ltwoD{ u_z, v_z }\right)^{\frac{1}{2}} .
\]
\begin{theorem}The space \( H^1_{D,f}(0,z_F) \) is a Hilbert space with respect to the norm
\[
\|v\|_{D,f}^1 = \left( \int_0^{z_F} f(z) v^2(z) \, dz + \int_0^{z_F} D(z) (v'(z))^2 \, dz \right)^{\frac{1}{2}}
\]
\end{theorem}

\begin{proof}

\noindent The space $H^1_{D,f}(0,z_F)$  has an associated inner product
\[
\langle u, v \rangle_{D,f} := \ltwof{ u, v } + \ltwoD {u_z, v_z }.
 \]
\noindent To prove that it is a Hilbert space, we let \( \{v_n\} \subset H^1_{D,f} \) be a Cauchy sequence in the induced norm. So, both \( \{v_n\} \) and \( \{v_n'\} \) are Cauchy in the respective weighted \( L^2_f \) and \( L^2_D \) spaces meaning that there exist \( v \in L^2_f \) and \( w \in L^2_D \) such that
\[
v_n \to v \text{ in } L^2_f(0,z_F), \quad v_n' \to w \text{ in } L^2_D(0,z_F).
\]
To prove that \(  w  := v_z\) over $(0,z_F)$,  
we let $\varphi \in C_c^{\infty}(0,z_F)$, i.e. $\varphi(z) \neq 0$ for $z\in [a,b]\subset (0,z_F)$ and zero elsewhere with $\varphi(z)\in C^{\infty}(a,b) \subset L^2(a,b) \subset L^2_f(a,b)$. 
\begin{eqnarray}
    \int_0^{z_f} v_n'(z) \varphi(z) dz &=& [v_n' \varphi]_0^{z_F} - \int_0^{z_f} v_n(z) \varphi'(z) dz = - \int_0^{z_F} v_n(z) \varphi'(z) dz\nonumber\\
    \implies \int_a^{b} v_n'(z) \varphi(z) dz &=&  - \int_a^{b} v_n(z) \varphi'(z) dz.\label{eq:vnloc}
\end{eqnarray}
Then, we have norm equivalence between  $L^2(a,b)$ and $L^2_f(a,b)$ since $f(z)>0$ over $[a,b]$. Thus, $v_n \rightarrow v$ in $L^2(a,b)$. Similarly, we have norm equivalence between  $L^2(a,b)$ and $L^2_D(a,b)$ since $D(z)>0$ over $[a,b]$. Thus, $v'_n \rightarrow w$ in $L^2(a,b)$. This implies we have norm equivalence between $H^1(a,b)$ and $H^1_{D,f}(a,b)$. Taking the limit of \eqref{eq:vnloc} gives
$$\int_a^{b} w(z) \varphi(z) dz =  - \int_a^{b} v(z) \varphi'(z) dz, $$
meaning that $w(z)$ is the weak derivative of $v(z)$ over any sub-interval $[a,b]$ of $(0,z_F)$. Thus,  \(  w  := v_z\) over $(0,z_F)$.  Then, \( v \in H^1_{D,f} \), and
\[
\|v_n - v\|_{D,f} \to 0.
\]
Hence \( H^1_{D,f}(0,z_F) \) is complete and therefore a Hilbert space.
\end{proof}

 \noindent Finally, the set of test functions is given by:
 $$\mathcal{T}=\{v\in H^1_{D,f}(0,z_F)\,|\, v(0)=0\} \subseteq H^1_{D,f} $$
 
 \begin{theorem}
 The set of test functions is a closed subspace of $H^1_{D,f}$, 
     $$\mathcal{T}=\{v\in H^1_{D,f}(0,z_F)\,|\, v(0)=0\} \subseteq H^1_{D,f}. $$
     i.e. a Hilbert Space with respect to the $H^1_{D,f}$ norm. 
 \end{theorem}
 \begin{proof}
Let \( \{v_n\} \subset \mathcal{T} \) be a Cauchy sequence.
Since \( H^1_{D,f}(0,z_F) \) is complete, there exists \( v \in  H^1_{D,f}(0,z_F) \) such that \( v_n \to v \) in \(  H^1_{D,f}(0,z_F) \), i.e. $\lim\limits_{n\rightarrow\infty}\Dfnorm{v_n-v} = 0$. It remains to prove that $v(0) = 0$.\vspace{2mm}

\noindent Let $0<\epsilon<z_F$. Since for all $z\in[0,z_F)$, $f(z)>0$ and $D(z)>0$, then we have norm equivalence between $H^1(0,\epsilon)$ and $H^1_{D,f}(0, \epsilon)$. Moreover, $H^1_{D,f}(0, z_F) \subset H^1_{D,f}(0, \epsilon),    $ thus \( v_n \to v \) in \(  H^1_{D,f}(0,\epsilon) \) \vspace{2mm}
 
\noindent By Trace theorem, $\forall v\in H^1(0,\epsilon)$ there exists a $c_T>0$ such that
\begin{equation}\label{eq:traceNE}
    \max\{|v(0)|, |v(\epsilon)|\} \;\leq \; c_T\,\|{v}\|_{H^1(0,\epsilon)}\;\leq \; \dfrac{c_T}{\gamma_1}\|{v}\|_{H^1_{D,f}(0,\epsilon)}
\end{equation}
using norm equivalence. 
Since $v_n - v \in H^1_{D,f}(0,z_F) \subset H^1_{D,f}(0,\epsilon)$, then using trace theorem \eqref{eq:traceNE} we get
\begin{eqnarray}
 \max\{|v_n(0)-v(0)|, |v_n(z_F) - v(z_F) |\} \;\leq \; \dfrac{c_T}{\gamma_1}\|{v_n-v}\|_{H^1_{D,f}(0,\epsilon)}.
\end{eqnarray}
Taking the limit as $n\rightarrow \infty$, and since $v_n(0)=0$, we get that $v(0) = 0$, i.e. $v\in\mathcal{T}$.
\noindent Thus $\mathcal{T}$ is a closed subspace. 
 Moreover, any closed subspace of a Hilbert space is itself a Hilbert space with respect to the  same norm.
 \end{proof}
 \subsection{Existence and Uniqueness}\label{sec:exist}
\noindent To apply Lions Theorem \ref{th:lions} to our firn problem \eqref{eq:system1} for proving the existence and uniqueness of its solution, we 
let
$$V=\mathcal{T}=\{v\in H^1_{D,f}(0,z_F)\,|\,v(0)=0\} \;\mbox{ and } \;H=L^2_f(0,z_f)  $$
Obviously, $V\subset H$, with the injection being continuous.\vspace{2mm}\\
It remains to prove the continuity and coercivity of the bilinear form $\mathcal{A}(.,.)$. Moreover, we prove  the Boundedness of the right-hand term $F(t, \phi) =  -\langle \rho^ {\text{atm}}(t))',\phi\rangle-\mathcal{A}(\rho^ {\text{atm}}(t),\phi)$ in $V$, which implies that there exists $\hat{F}(t) \in L^2(0,T;V^*)$ such that $F(t, \phi) = \inner{\hat{F}(t), \phi}_{D,f}$.

\noindent It should be noted that if $f(z)$ and $D(z)$ are strictly positive functions over $[0,z_F]$, then this implies  we have norm equivalences, where $\exists\; g_1,g_2, \hat{g}_1, \hat{g}_2, b_1, b_2 >0$ such that
 \begin{eqnarray}
 &&  g_1\ltwonorm{v}  \hspace{0.2cm}  \; \leq\;\;\;\fnorm{v} \hspace{0.2cm}  \leq\;\;g_2 \;\ltwonorm{v} \label{eq:eqf}\\
  &&  \hat{g}_1\ltwonorm{v}  \hspace{0.2cm}  \; \leq\;\;\;\Dnorm{v} \hspace{0.2cm}  \leq\;\;\hat{g}_2 \;\ltwonorm{v} \label{eq:eqD}\\
  &&  b_1\fnorm{v}  \hspace{0.2cm}  \; \leq\;\;\;\Dnorm{v} \hspace{0.2cm}  \leq\;\;b_2 \;\fnorm{v}. \label{eq:eqDf}
    \end{eqnarray}
The norm equivalences and a Weighted Poincare Inequality  allow us to complete the proof. However, in practice $f$ and $D$ may be degenerate at $z_f$, hence we loose norm equivalence. Thus, we assume that $f(z_F)\geq 0$ and $D(z_F)\geq 0$ (assumptions \eqref{A10}, \eqref{A20}) and impose in addition to assumption
\eqref{A2pp0},  that $f(z)$ and $D(z)$ have { similar behavior at $z_F$,} as in assumption \eqref{A2p2}.\vspace{2mm}

\noindent We proceed by proving the Weighted Poincare Inequality  for all functions in $\mathcal{T}$ (Corollary \ref{cor:wpi}) by using the Hardy-type inequality (Theorem \ref{th:WPICa}) for absolutely continuous functions on every compact subinterval, which includes  $\mathcal{T}$ (Lemma \ref{lem:tc}). This allows us to prove Theorem \ref{th:4g}.
\begin{remark}\label{rem}
    Let $u(z) \in C_a(0,z_F) $, where $C_a(0,z_F)$ is the set of all functions absolutely continuous on every compact subinterval $[a,b]\subset (0,z_F)$, satisfying $\lim\limits_{x\rightarrow 0} u(x) = 0$.\\Then there exists a function $g \in L^1(0,z_F)$ such that $$u(z) = \int_0^z g(x)dx$$ a.e. on $(0,z_F)$. The converse is also true.
\end{remark}

\begin{theorem}[Weighted Poincare Inequality]\label{th:WPICa}
  Let $v,w \in W(0,z_F)$, where $W(0,z_F)$ is the set of weight functions that are measurable, nonnegative, and finite almost everywhere on $(0,z_F)$. Then,  
  \begin{equation}\label{eq:WPI}
  \|u\|_v = \left(\int_0^{z_F} u^2(z)\;v(z)dz\right)^{\frac{1}{2}} \leq C \left(\int_0^{z_F} (u'(z))^2\;w(z)dz\right)^{\frac{1}{2}} = C \|u'\|_w
  \end{equation}
  holds for every $u(z) \in C_a(0,z_F)$ if and only if
  \begin{eqnarray} \label{condition:WPI}
      B = \sup_{z\in(0,z_F)} \left({ \int_z^{z_F} v(x)dx}\right)^{\frac{1}{2}}\left( \int_0^z \dfrac{1}{w(x)}dx \right)^{\frac{1}{2}} < \infty,
\end{eqnarray}
Moreover, the best constant $C$ in \eqref{eq:WPI} satisfies $B \leq C \leq 2B$
\end{theorem}
\begin{proof} By \cite[Theorem 1.14]{OK90}
for $p=q=p'=2$, $a=0, b=z_F$.  
\end{proof}

\begin{lemma}\label{lem:tc}
    $\mathcal{T}=\{v\in H^1_{D,f}(0,z_F)\,|\, v(0)=0\} \subseteq C_a(0,z_F) $
\end{lemma}
\begin{proof} Let $v \in \mathcal{T}$, then $v(0) = 0$ and $v\in H^1_{D,f}(0,z_F)$.
    Since $f>0, D>0$ on $(0,z_F)$, then for any $[a,b]\subset (0,z_F)$ there exists $f_{a,b},  D_{a,b} >0$ such that 
    \begin{eqnarray}
        &&f_{a,b} \int_a^b  v^2 dz\;\leq\;\int_a^b f(z) v^2 dz\;\leq\; \fnorm{v}^2 = \int_0^{z_F} f(z) v^2 dz\\
         &&D_{a,b} \int_a^b  v_z^2 dz\;\leq\; \int_a^b D(z) v_z^2 dz\;\leq\; \Dnorm{v_z}^2 = \int_0^{z_F} D(z) v_z^2 dz
    \end{eqnarray}
    Therefore,   
    $H^1_{D,f}(0,z_f) \subset H^1(a,b)  $ for all $[a,b] \subset (0,z_F)$. 
    Then, for all $[a,b]\subset (0,z_F)$ and for all  $u\in \mathcal{T}$,  $ u-weak \in L^2(a,b)  \subset L^1(a,b)$ thus there exists a $g(z) \in L^1(a,b)$ such that\\ $u(z)-u(a) = \int_a^z g(t)\;dt$, implying that $u$ is absolutely continuous over $[a,b]$, hence $u \in C_a(0,z_F)$. 
\end{proof}

  \begin{corollary}[Weighted Poincare Inequality over $\mathcal{T}$]\label{cor:wpi}
  Let $u \in \mathcal{T} \subset C_a(0,z_F)$. Then, 
  if 
  \begin{eqnarray}  \label{cond:WI}
B_0=\displaystyle\sup_{z\in(0,z_F)} \left((z_F-z)\int_0^z \dfrac{1}{D(x)}dx \right)^{\frac{1}{2}} <\infty,    
  \end{eqnarray}  the below inequalities hold
  \begin{eqnarray}
  &&\ltwonorm{u}^2 = \int_0^{z_F} u^2(z)dz \leq C_p^2 \int_0^{z_F} (u'(z))^2\;D(z)dz = C_p^2 \|u'\|_D^2 \label{eq:wp11}\\
  &&\fnorm{u}^2 \leq \|f\|_{\infty} \ltwonorm{u}^2 \leq  \|f\|_{\infty} C_p^2 \|u'\|_D^2\label{eq:wp12}
  \end{eqnarray}
for $C_p<2B_0$.
  \end{corollary}
  \begin{lemma}Assume that the condition \eqref{cond:WI} is satisfied. 
Then the  following inequality 
\begin{eqnarray}
   \sqrt{f(z)} |v(z)| \leq  C_0 \|v\|_{D, f}, \qquad z\in [0, z_F],\label{eq:zf}
\end{eqnarray}
holds for all $ v \in \mathcal T$, where $C_0$ is  some positive constant that only depends on $f, \,D$ and $z_F$.
\end{lemma}
  \begin{proof}
Since $v \in T\subset C_a(0, z_F)$, and we have  
$$
\sqrt{f(z)}  v(z) =  \sqrt{f(z)} \int_0^z v_z(t) dt, \quad \forall
 z\in [0, z_F].
$$
Therefore
$$
\sqrt{f(z)}  |v(z)| \leq  \left\|\frac{\sqrt f}{ \sqrt D} \right\|_{\infty} \left(D(z)\int_0^z \frac{1}{D(t)}dt\right)^{\frac{1}{2}} \|v_z\|_{D}
, \quad \forall
 z\in [0, z_F].
$$
We claim that    
$$
 D(z)\int_0^z \frac{1}{D(t)}dt <\infty, \quad \forall z\in [0, z_F].
$$
Indeed, if $D(z_F) \not= 0$, we get 
$$
D(z)\int_0^z \frac{1}{D(t)}dt  \leq \|D\|_{{\infty}} \int_0^{z_F} \frac{1}{D(t)}dt< \infty.
$$
Assume now that $D(z_F) = 0$. Since $D$ is Lipschitz continuous \eqref{A2pp0}, then there exists $L > 0$ such that
$$D(z) = D(z) -D(z_F)  \leq L\; |z-z_F|  .$$ Consequently 
$$
D(z)\int_0^z \frac{1}{D(t)}dt \leq L\;B^2 <\infty. 
$$
The proof is then accomplished by taking 
$C_0=\displaystyle { \left\|\frac{\sqrt f}{ \sqrt D} \right\|_{\infty} }\|D\|_{\infty} \int_0^{z_F} \frac{1}{D(t)}dt$ if $D(z_F) \neq 0$, and\\ $C_0 = \displaystyle \left\|\frac{\sqrt f}{ \sqrt D} \right\|_{\infty}L\; B^2$ if $D(z_F) = 0$.
  \end{proof}

\begin{theorem}\label{th:4g} Under assumptions  \eqref{A10}, \eqref{A20}, \eqref{A2pp0}, \eqref{eq:rhoatm}, 
  \eqref{A2p2}, and \eqref{cond:WI}, then the bilinear form $$ \mathcal{A}(v,w)=\langle D v_z, w_z \rangle -\langle  (\mathcal{M} D+f\mathcal{F})v, w_z \rangle + \mathcal{G} \langle v,w \rangle + \mathcal{F} f(z_F) v(z_F,t) w(z_F)$$ 
satisfies:
\begin{enumerate}
    \item $|\mathcal{A}(v,w)|\leq C\Dfnorm{v}.\Dfnorm{w}$,\; $\forall v\in H^1_{D,f}$
    \item $\mathcal{A}(v,v)\geq C_1\Dfnorm{v}^2-C_2\fnorm{v}^2 $,\; $\forall v\in \mathcal{T}$
    \item {$\displaystyle \left(\int_0^T | \ltwoinner{\rho^ {\text{atm}}(t))',v}+\mathcal{A}(\rho^ {\text{atm}}(t),v)|^2 dt\right)^{\frac{1}{2}} \leq C_3 
    \Dfnorm{v}$, \; $\forall v\in H^1_{D,f}$.}
\end{enumerate}
for $0<C,\,C_1,\,C_2,\,C_3<\infty$.
\end{theorem}
\begin{proof}
We prove the three properties of the bilinear form \( \mathcal{A}(\cdot,\cdot) \) and the right-hand side.
\begin{enumerate}
    \item {\bf Continuity on \( H^1_{D,f} \)}:
We estimate each term in \( \mathcal{A}(v,w) \) using Cauchy--Schwarz and the weighted norms definitions and equivalence relations.
\begin{eqnarray}
\big|\langle D v_z, w_z\rangle\big|
=\big|\langle D^{\frac{1}{2}} v_z, D^{\frac{1}{2}} w_z\rangle\big| 
&\le& \ltwonorm{ D^{\frac{1}{2}} v_z} \ltwonorm{ D^{\frac{1}{2}} w_z} = \Dnorm{v_z} \Dnorm{w_z} \nonumber\\
&\leq& \|v\|_{D,f}\,\|w\|_{D,f}.\label{eq:1t}
   \end{eqnarray}
\noindent Since $ \|v\|_D \leq \|D^{\frac{1}{2}}\|_{\infty}\|v\|_2$  \;and by \eqref{eq:wp11} $\exists C_p>0$ such that $\ltwonorm{v} \leq C_p \Dnorm{v_z}$, we get
\begin{eqnarray}
\mathcal{M}\,\big|\langle Dv,w_z\rangle\big|  &=& \mathcal{M}\,\big|\langle D^{\frac{1}{2}}v,D^{\frac{1}{2}}w_z\rangle\big| \, \le \, \mathcal{M}\|D^{\frac{1}{2}}v\|_2\;\|D^{\frac{1}{2}}w_z\|_2 = \mathcal{M}\,\|v\|_D\|w_z\|_D    \nonumber\\
&\leq & \mathcal{M} \|v\|_D\|w\|_{D,f}\;\leq \; \mathcal{M} \,\|{D}^\frac{1}{2}\|_{\infty}\,\,\|v\|_{2}\|w\|_{D,f}\nonumber\\
&\leq & C_p \mathcal{M} \,\|{D}^\frac{1}{2}\|_{\infty}\,\,\|v\|_{D,f}\|w\|_{D,f}\label{eq:2t}
\end{eqnarray}
Similarly,  we get
\begin{eqnarray}
    | \langle f\mathcal{F} v, w_z \rangle| = \mathcal{F} \left|\inner{ \sqrt{\dfrac{f}{D}} f^{\frac{1}{2}}v, D^{\frac{1}{2}}w_z }\right|   &\leq& \mathcal{F}\linftynorm{\sqrt{{f}/{D}}}   \fnorm{v} \Dnorm{w_z} \nonumber\\
    &\leq& \mathcal{F} \linftynorm{\sqrt{{f}/{D}}}  \|v\|_{D,f} \|w\|_{D,f}\qquad\\
    \mathcal{G} \inner{ v, w } \;\le\; \mathcal{G} \, \ltwonorm{v}  \, \ltwonorm{w} \;\leq\;\mathcal{G}C_p^2 \|v_z\|_D\|w_z\|_D &\leq& \mathcal{G} C_p^2 \|v\|_{D,f} \, \|w\|_{D,f} \, .
    \end{eqnarray} 
By \eqref{eq:zf} we have that $\sqrt{f(z_F)}|v(z_F,t)| \leq 
C_0 \Dfnorm{v}$ and $\sqrt{f(z_F)}|w(z_F)| \leq 
C_0 \Dfnorm{w}$ which implies that
\begin{eqnarray}
|\mathcal{F} f(z_F) v(z_F,t) w(z_F)| 
= \mathcal{F} f(z_F) |v(z_F,t)| \,|w(z_F)| 
&\le& \mathcal{F} \,C_0^2\, \|v\|_{D,f} \,\, \|w\|_{D,f} 
\end{eqnarray}
\noindent Thus,
\begin{eqnarray}
|\mathcal{A}(v, w)| 
&\le & |\langle D v_z, w_z \rangle| + \mathcal{M} |\langle D v, w_z \rangle| 
   + |\langle f \mathcal{F} v, w_z \rangle| + \mathcal{G} |\langle v, w \rangle| + \mathcal{F} |f(z_F) v(z_F,t) w(z_F)|
  \nonumber \\
&\le &\left(1+ C_p \mathcal{M} \,\|{D}^\frac{1}{2}\|_{\infty} + {\mathcal{G}}\,C_p^2+\mathcal{F} \linftynorm{\sqrt{\dfrac{f}{D}}}+\mathcal{F} \,C_0^2\right)\;\|v\|_{D,f} \|w\|_{D,f} 
\end{eqnarray}
\noindent where 
$
C := 1+ C_p \mathcal{M} \,\|{D}^\frac{1}{2}\|_{\infty} + {\mathcal{G}}\,C_p^2+\mathcal{F} \linftynorm{\sqrt{f/D}}+\mathcal{F} \,C_0^2
>1$ and $C<\infty$ by assumptions \eqref{A2p2} and \eqref{A20}. Therefore, the bilinear form $\mathcal{A}$ is continuous on $H^1_{D,f}$, i.e.,
\[
|\mathcal{A}(v, w)| \le C \, \|v\|_{D,f} \, \|w\|_{D,f}, \quad \forall v, w \in H^1_{D,f}.
\]
\item {\bf Coercivity:} Let \( v \in \mathcal{T} \subset H^1_{D,f} \) with $\Dnorm{v_z}^2 = \Dfnorm{v}^2 - \fnorm{v}^2$, then
\begin{eqnarray}
\mathcal{A}(v,v) 
&=& \Dnorm{v_z}^2 +\mathcal{G} \|v\|_2^2- \mathcal{M}\langle   D  v, v_z \rangle - \mathcal{F}\langle   f  v, v_z \rangle + \mathcal{F} f(z_F) v(z_F,t) w(z_F) \nonumber\\
&\geq &\Dfnorm{v}^2 - \fnorm{v}^2 - \mathcal{M}\langle   D  v, v_z \rangle - \mathcal{F}\langle   f  v, v_z \rangle
\end{eqnarray}
and by Cauchy Schwarz inequality
\begin{eqnarray}
    \mathcal{M}\inner{  D  v, v_z } &=&  \mathcal{M}\inner{  \dfrac{D^{\frac{1}{2}}}{f^{\frac{1}{2}}}  f^{\frac{1}{2}}v, D^{\frac{1}{2}}v_z } \leq \mathcal{M} \linftynorm{\sqrt{D/f}}\fnorm{v} \Dnorm{v_z} \nonumber\\
    \mathcal{F}\inner{  f  v, v_z}&=& \mathcal{F} \inner{  \dfrac{f^{\frac{1}{2}}}{D^{\frac{1}{2}}} f^{\frac{1}{2}}v, D^{\frac{1}{2}}v_z } \;\leq \;  \mathcal{F} \linftynorm{\sqrt{f/D}}\fnorm{v} \Dnorm{v_z}    \nonumber\\
    \therefore \inner{ (\mathcal{M} D + f \mathcal{F} ) v, v_z }
    &\leq & (\mathcal{M} \,{b}_2 \,  + \mathcal{F}\,\tilde{b}_2)\,  \,\fnorm{v} \,\Dnorm{v_z} \;\leq\;c\,  \,\fnorm{v} \,\Dfnorm{v}  \nonumber\\
    &\leq &c\left( \dfrac{\epsilon}{2} \Dfnorm{v}^2 + \frac{1}{2 \epsilon} \|v\|_f^2 \right) \qquad \qquad  \label{eqMDFcoer2}
\end{eqnarray}
\noindent by applying Young's inequality  for any \(\epsilon > 0\) 
where $c =  (\mathcal{M} \,{b}_2 \,  + \mathcal{F}\,\tilde{b}_2)$, ${b}_2 = \linftynorm{\sqrt{D/f}} < \infty$ and $ \tilde{b}_2 = \linftynorm{\sqrt{f/D}}<\infty $ by \eqref{A2p2}.
\noindent By replacing \eqref{eqMDFcoer2} in the bilinear form, we obtain 
\begin{eqnarray}
\mathcal{A}(v,v) &\ge&  \|v_z\|_D^2  - \langle ( \mathcal{M} D + f \mathcal{F} ) v, v_z \rangle \nonumber\\
&\geq & (1 - c\epsilon/2 ) \Dfnorm{v}^2- (1+ c/(2\epsilon))  \|v\|_f^2  \qquad
\end{eqnarray}
\noindent Let \(\epsilon > 0\) be sufficiently small so that $1 - \epsilon  c/2 > 0$, which concludes  the proof.
\item {\bf Boundedness of the forcing term:}

Let \( v \in H^1_{D,f}\), we get by continuity of $\mathcal{A}$ and \eqref{eq:wp11} then:
\begin{eqnarray}
|\langle (\rho^ {\text{atm}})'(t), v \rangle + \mathcal{A}(\rho^ {\text{atm}}(t), v)| &\leq& |\langle (\rho^ {\text{atm}})'(t), v \rangle| + |\mathcal{A}(\rho^ {\text{atm}}(t), v)| \nonumber\\
&\leq & \|(\rho^ {\text{atm}})'(t) \|_2\,\|v \|_2 + C \|\rho^ {\text{atm}}(t)\|_{D,f} \|v\|_{D,f}\nonumber\\
&\leq &C_p\ltwonorm{(\rho^ {\text{atm}})'(t) } \, \Dnorm{v_z} + C \|\rho^ {\text{atm}}(t)\|_{f} \|v\|_{D,f}\nonumber\\
&=& C_p z_F|(\rho^ {\text{atm}})'(t)|\Dnorm{v_z} +C|\rho^ {\text{atm}}(t)|\ltwonorm{f^{\frac{1}{2}}}^2\|v\|_{D,f}\nonumber\\
&\leq & \tilde{C_2}(|(\rho^ {\text{atm}})'(t)| + | \rho^ {\text{atm}}(t)|) \Dfnorm{v}
\end{eqnarray}
where $\tilde{C_2}= \max\left\{C_p \,z_F,C\ltwonorm{f^{\frac{1}{2}}}^2\right\}$. 
Then, for $C_3 = \sqrt{3}\; \tilde{C_2} \|\rho^{\text{atm}}\|_{H^1(0,T)}<\infty$ we have
\begin{eqnarray}
\hspace{-2mm}\int_0^T | \ltwoinner{\rho^ {\text{atm}}(t))',v}+\mathcal{A}(\rho^ {\text{atm}}(t),v)|^2\; dt &\leq & 3\, \tilde{C_2}^2 \|\rho^{\text{atm}}\|^2_{H^1(0,T)} \Dfnorm{v}^2 \;=\; C_3^2\;\Dfnorm{v}^2 \nonumber\end{eqnarray}
Moreover, $F(t,\phi) = \langle (\rho^ {\text{atm}})'(t), \phi \rangle + \mathcal{A}(\rho^ {\text{atm}}(t), \phi)$ is linear in $\phi$ by the bilinearity of the inner product and $\mathcal{A}(.,.)$.
\end{enumerate}
Hence, all three properties in the theorem hold which implies that \eqref{eq::system} or \eqref{eq:system1} has a unique solution by Lion's Theorem \ref{th:lions}.
\end{proof}
\section{Discretization in space and time}\label{sec:discch}
We first rescale \eqref{eq:Witrant} to the unit square, then we discretize the obtained variational problem in time using a finite diﬀerence Euler-implicit scheme, followed by space discretization using $\mathbb{P}1$ finite element method. Note that we consider $f(z_F) \geq 0$.
\subsection{Rescaling} We rescale both the time interval $[0,T_e]$ and the space interval $[0,z_f]$ to the unit interval $[0,1]$ by letting $\tilde{t} = t/T_e$, $\tilde{z} = z/z_F$, $\tilde{\rho}\, (\tilde{t} ,\tilde{z} ) = \rho\, (t,z)$, $\tilde{f}(\tilde{z} ) =f(z)$ and $\tilde{D}(\tilde{z} ) =D(z)$, where by the chain rule we have: $$\dfrac{\partial  \rho }{\partial t}= \dfrac{1}{T_e}\dfrac{\partial \tilde{\rho}}{\partial \tilde{t}}, \;\;\dfrac{\partial  \rho }{\partial z}= \dfrac{1}{z_F}\dfrac{\partial \tilde{\rho}}{\partial \tilde{z}},
\;\;\dfrac{\partial ^2  \rho }{\partial z^2}= \dfrac{1}{z_F^2}\dfrac{\partial ^2 \tilde{\rho}}{\partial \tilde{z}^2}, \;\; \dfrac{\partial  f }{\partial z}= \dfrac{1}{z_F}\dfrac{\partial \tilde{f}}{\partial \tilde{z}},  \;\; \dfrac{\partial  D }{\partial z}= \dfrac{1}{z_F}\dfrac{\partial \tilde{D}}{\partial \tilde{z}} .$$
Then, system \eqref{eq:Witrant} becomes for $\tilde{t} \in [0,1]$ and $
\tilde{z} \in [0,1]$

\begin{equation} \label{eq:RescWitrant}
\begin{cases}
{\dfrac{1}{T_e}\dfrac{\partial}{\partial \tilde{t}}[\tilde{\rho} \tilde{f}] + \dfrac{1}{z_F}\dfrac{\partial}{\partial \tilde{z}}[\tilde{\rho} \tilde{f}\mathcal{F}] + \tilde{\rho} \mathcal{G} = \dfrac{1}{z_F}\dfrac{\partial}{\partial \tilde{z}}\left[\tilde{D}\left(\dfrac{1}{z_F}\dfrac{\partial\tilde{\rho}}{\partial \tilde{z}} - \tilde{\rho} \mathcal{M}\right)\right]}, \\
{\tilde{\rho}(0,\tilde{t}) = \tilde{\rho}^\text{atm}(t)}, \\
{\tilde{D}(1)\left(\dfrac{1}{z_F}\dfrac{\partial\tilde{\rho}}{\partial \tilde{z}}(1,\tilde{t}) - \mathcal{M} \rho(1,\tilde{t})\right) = 0},  \\
{\tilde{\rho}(\tilde{z},0) = 0}
\end{cases}
\end{equation}
Replacing in \eqref{eq:RescWitrant} $\tilde{\rho}$ by $\rho$,  $\tilde{t}$ by $t \in [0,1]$,  $\tilde{z}$ by $z \in [0,1]$, $\tilde{D}$ by $D$, we get the following semi-variational form
\begin{eqnarray}
&&\dfrac{1}{T_e}\int_0^{1} f \rho_t \phi \hspace{0.1cm} dz+ \dfrac{\mathcal{F}}{z_F}\int_0^{1} (f \rho )_z \phi  \hspace{0.1cm} dz+ \mathcal{G} \int_0^{1} \rho \phi  \hspace{0.1cm} dz = \dfrac{1}{z_F}\int_0^{1} [D (\dfrac{\rho_z}{z_F} - \rho \mathcal{M})]_z \phi dz\nonumber\\
&&\dfrac{1}{T_e}\int_0^{1} f \rho_t \phi \hspace{0.1cm} dz+ \dfrac{\mathcal{F}}{z_F}f(1) \rho(1,t) \phi(1) - \dfrac{\mathcal{F}}{z_F}\int_0^{1} f \rho  \phi_z  \hspace{0.1cm} dz+ \mathcal{G} \int_0^{1} \rho \phi  \hspace{0.1cm} dz =- \dfrac{1}{z_F}\int_0^{1} D (\dfrac{\rho_z}{z_F} - \rho \mathcal{M}) \phi_z dz\nonumber\\
&&\therefore \langle f\rho_t , \phi \rangle + \dfrac{T_e}{z_F} \mathcal{F} f(1) \rho(1,t) \phi(1) -  \dfrac{T_e}{z_F}\mathcal{F} \langle f \rho , \phi_z \rangle + T_e\mathcal{G} \langle \rho, \phi \rangle +\dfrac{T_e}{z_F^2}\langle D \rho_z, \phi_z \rangle - \dfrac{T_e}{z_F}\mathcal{M} \langle D \rho, \phi_z \rangle  = 0.\qquad \label{eq:rescvar}
\end{eqnarray}
For ${t} \in [0,1]$ and $
{z} \in [0,1]$,  let the rescaled bilinear form $\mathcal{A}$ be given by
\begin{equation}\label{eq:rescbil}
\mathcal{A}(\rho, \phi) = \dfrac{T_e}{z_F^2}\langle D \rho_z, \phi_z \rangle -  \dfrac{T_e}{z_F}\mathcal{F} \langle f \rho , \phi_z \rangle  - \dfrac{T_e}{z_F}\mathcal{M} \langle D \rho, \phi_z \rangle+ T_e\mathcal{G} \langle \rho, \phi \rangle +\dfrac{T_e}{z_F} \mathcal{F} f(1) \rho(1,t) \phi(1)\quad
\end{equation}
then \eqref{eq:rescvar} becomes 
\begin{equation}\label{FirnV}
\ltwof{ \rho_t , \phi} + \mathcal{A}(\rho,\phi) =0
\end{equation}
\subsection{Euler-Implicit Time Discretization}
By using Euler-Implicit time discretization on \eqref{FirnV},  one  obtains the following 
scheme.
 \begin{equation}\label{Firn-tau23}
\left\{\begin{array}{ll}
\ltwof{\rho(\cdot, t+\Delta t)-\rho(\cdot,t),\phi} =-{\Delta t}\;  \mathcal{A}(\rho(\cdot,t+\Delta t), \phi)\, ,&\\
 \rho(0,t) = \rho^ {\text{atm}}(t) \,, &\\
 \rho(z,0) = 0\, . &\\
\end{array}\right.
\end{equation}
 \subsection{Finite Element Space Discretization}\label{sec:disc}
 Let 
$\mathcal{N}=\{z_i  \;|\;i=1,2,...,n\}$ be the set of nodes based on the partition of $(0,1)$ with $$0=z_1<z_2<...<z_n= 1,$$  
and $\mathcal{E} = \{ E_j = [z_j,z_{j+1}] \;|\; j=1,2,..,n-1\}$ the resulting set of elements.\\
The $\mathbb{P}_1$ finite element subspace $X_n$ of $H^1(0,1)$ is given by:
 $$X_n=\{v\in C(0,1)\,|\,v\mbox{ restricted to }E_j\in\mathbb{P}_1,\,j=1,2..,n-1\} \subset H^1_{D,f}(0,1) \, . $$ 
 Consistently, we define $$X_{n,d} = X_n \cap H^1_{D,f}(0,1)\, .$$
 For that purpose, we let ${B}_n=\{\varphi_i|\,i=1,2,...n\}$ be a finite element basis of functions with compact support in $(0,1)$, i.e.,: 
 \begin{equation}\label{FEdef}\forall v_n \in X_n:\,v_n(z)=\sum\limits_{i=1}^n{V_{i}\varphi_i(z)},\;\;\;\; V_i=v_n(z_i),\end{equation} where 
 $\varphi_1(z) = \begin{cases} \dfrac{z_2 - z}{z_2 - z_1}, & z_1\leq z \leq z_2\\
 0,& \mbox{otherwise} \end{cases}$, \qquad $\varphi_n(z) = \begin{cases} \dfrac{z - z_{n-1}}{z_n - z_{n-1}},& z_{n-1}\leq z \leq z_n\\
 0,& \mbox{otherwise} \end{cases}$, \\and  $\varphi_i(z) = 
 \begin{cases} 
 \dfrac{z - z_{i-1}}{z_i - z_{i-1}},& z_{i-1}\leq z \leq z_i\\
 \dfrac{z_{i+1} - z}{z_{i+1} - z_i}, & z_i\leq z \leq z_{i+1}\\
 0,& \mbox{otherwise} \end{cases} $ for $i = 2,.., n-1$\, .\\\\
We hence obtain the following fully implicit {Computational Model}. \vspace{2mm}\\
Given $\rho_n(t) \in X_{n,d} +\{\rho^ {\text{atm}}\}$, one seeks: 
 \begin{equation}\label{Firn-tau2}
\left\{\begin{array}{ll}
\rho_n(t+\Delta t)\in  X_{n,d} +\{\rho^ {\text{atm}}(t)\}
&\\
\ltwof{\rho_n(t+\Delta t),\phi} +{\Delta t}  \mathcal{A}(\rho_n(t+\Delta t), \phi)= \ltwof{\rho_n(t),\phi} ,& \quad\forall \phi \in X_{n,d}, \;\; \forall t>0\\
 \rho_n(0) = 0&\\
\end{array}\right.\vspace{-2mm}
\end{equation} 
where \vspace{-8mm}
\begin{eqnarray}
\rho_n(t) &=& \rho_{n,d}(t) +  \rho^ {\text{atm}}(t) \varphi_1(z),\hspace{7cm}\nonumber\\
\rho_{n,d}(t) &=& \sum\limits_{i=2}^n{\rho(z_{i},t)\varphi_i(z)} \, .\nonumber
\end{eqnarray}
Thus, \eqref{Firn-tau2} is equivalent to \eqref{Firn-tau33}.\\ 
Given $\rho_{n,d}(t) \in X_{n,d}$, one seeks $\forall t\in(0,1]$, $\forall \phi \in X_{n,d}$: 
 \begin{equation}\label{Firn-tau33}
\hspace{-3mm}\left\{\begin{array}{lcl}
\rho_{n,d}(t+\Delta t)\in  X_{n,d} 
&&\\
\ltwof{\rho_{n,d}(t+\Delta t),\phi} +{\Delta t}  \mathcal{A}(\rho_{n,d}(t+\Delta t),\phi)=\ltwof{\rho_{n,d}(t),\phi}-{\Delta t} \mathcal{A}(\rho^ {\text{atm}}(t+\Delta t)\varphi_1, \phi)&&\\
\qquad \hspace{7cm}-\ltwof{(\rho^ {\text{atm}}(t+\Delta t)-\rho^ {\text{atm}}(t)) \varphi_1,\phi}&&\\
 \rho_{n,d}(0) = 0.&&
\end{array}\right.
\end{equation}
where the bilinear form $\mathcal{A}(.,.)$ is defined in \eqref{eq:rescbil}
\subsection{Matrix Form of the Discrete System}\label{sec:mat}
Let $\phi = \varphi_j$ for $j=2,..,n$ in \eqref{Firn-tau33} and define the vector $\Lambda(t) = [\rho(z_{2},t), \, \rho(z_{3},t), \, \cdots ,\rho(z_{n},t)]^T$ of length $n-1$, then \eqref{Firn-tau33} can be written in Matrix form  
\begin{equation}\label{eq:mat2}
\left\{\begin{array}{lcl}
\left[M_f+T_e\Delta t\; C\right]\Lambda(t+\Delta t) &=&  M_f\Lambda(t) - T_e\Delta t \;b\\
\Lambda(0)= 0
\end{array}\right.\vspace{-1mm}
\end{equation}
where \vspace{-1mm}
\begin{itemize}
 \item  $ C = {\mathcal{G}} M+\dfrac{1}{z_F^2}S(D)-\dfrac{\mathcal{M}}{z_F}A(D)+\dfrac{1}{z_F}(B-K_f) \vspace{-2mm}$
\item $   T_e\,\Delta t \,b =  v_1(t)+T_e\,\Delta t\, v_3(t); \quad i.e. \quad b = \dfrac{1}{T_e\,\Delta t} v_1(t)+ v_3(t)$ are vectors of length $n-1$ with $v_1(t)$ and $v_3(t)$ defined in \eqref{eq:v1} and \eqref{eq:v3} respectively.
\item $M_f, M, S(D), A(D), K_f $ are $(n-1)\times(n-1)$ matrices whose entries are respectively for $i,j = 1,2,\cdots, n-1$:\vspace{1mm}\\ $
(M_f)_{i,j} = \ltwof{ \varphi_{i+1},\varphi_{j+1}},\;\; M_{i,j} = \ltwoinner{\varphi_{i+1} ,\varphi_{j+1}}, \;\;S_{i,j} = \ltwoinner{D \varphi_{i+1}',\varphi_{j+1}'},$\vspace{1mm}\\ $K_{i,j} =\mathcal{F} \ltwof{ \varphi_{i+1}' ,\varphi_{j+1}},\;\; A_{i,j} = \ltwoinner{D \varphi_{i+1}', \varphi_{j+1}}.$ 
\item $B$ is an $(n-1)\times(n-1)$ zero matrix with $B(n-1,n-1) = \mathcal{F}f_n$ 
\end{itemize}
by noting that for $j=2,..,n$:
\begin{itemize}
\item $\ltwof{\rho_{n,d}(t+\Delta t),\varphi_j} = \sum\limits_{i=2}^n \rho(z_{i},t+\Delta t)\ltwof{ \varphi_i,\varphi_j}$ is equivalent to $M_f \Lambda(t+\Delta t)$. 
\item $\ltwof{\rho_{n,d}(t+\Delta t),\varphi_j} = \sum\limits_{i=2}^n \rho(z_{i},t)\ltwof{\varphi_i,\varphi_j}$ is similarly equivalent to $M_f \Lambda(t)$.
\item $(\rho^ {\text{atm}}(t+\Delta t)-\rho^ {\text{atm}}(t))\ltwof{\varphi_1,\varphi_j}$ is equivalent to the  vector of length $n-1$, \begin{equation}\label{eq:v1} v_1(t) = (\rho^ {\text{atm}}(t+\Delta t)-\rho^ {\text{atm}}(t))\ltwof{\varphi_1,\varphi_2} e_1 
\end{equation} where $e_1 =[1,\, 0\, \cdots , \, 0]^T$.
\item {\color{white}.} \vspace{-10.5mm} \hspace{-5mm} \begin{eqnarray}
\hspace{-4mm}  \mathcal{A}(\rho_{n,d}(t+\Delta t),\varphi_j) &=&\mathcal{A}\left(\sum\limits_{i=2}^n{\rho(z_{i},t+\Delta t)\varphi_i} , \varphi_j\right) \nonumber\\
&=& {T_e\mathcal{G}}\sum\limits_{i=2}^n\rho(z_{i},t+\Delta t)\ltwoinner{{\varphi_i} ,\varphi_j}  - \dfrac{T_e}{z_F}\sum\limits_{i=2}^n\mathcal{F} \rho(z_{i},t+\Delta t) \ltwof{{\varphi_i} ,\varphi_j'} 
 \nonumber\\
&&+ 
\dfrac{T_e}{z_F^2} \sum\limits_{i=2}^n\rho(z_{i},t+\Delta t)\ltwoinner{D \varphi_i' ,\varphi_j'}
+ \,  \dfrac{T_e\mathcal{F} }{z_F}f_n\,\varphi_j(z_n)\,\rho(z_n,t+\Delta t) \nonumber\\ 
&& - \dfrac{T_e\mathcal{M}}{z_F}\sum\limits_{i=2}^n\rho(z_{i},t+\Delta t)\ltwoinner{D \varphi_i ,\varphi_j'}  
,  \nonumber\end{eqnarray}
is equivalent to
 \begin{eqnarray}&&T_e\left({\mathcal{G}}M+\frac{1}{z_F^2}S(D)-\dfrac{1}{z_F}K_f-\frac{\mathcal{M}}{z_F}A(D)\right)\Lambda(t+\Delta t) + \dfrac{T_e}{z_F}v_2(t+\Delta t)\qquad\nonumber\\ 
 &\iff& T_e\left(\mathcal{G}M+\frac{1}{z_F^2}S(D)-\dfrac{1}{z_F}K_f-\frac{\mathcal{M}}{z_F}A(D)+\dfrac{1}{z_F}B\right)\Lambda(t+\Delta t)
\label{bilinA}\end{eqnarray} 
 where $v_2(t+\Delta t) = \mathcal{F} f_n \rho({z_n},t+\Delta t) e_{n-1} = B\Lambda(t+\Delta t)$ is an $(n-1) \times 1$ vector of zeros except the last entry.
\item {\color{white}.}\vspace{-0.95cm} 
\begin{eqnarray}
\mathcal{A}\left(\rho^ {\text{atm}}(t+\Delta t)\varphi_1 , \varphi_j\right) &=& T_e\rho^ {\text{atm}}(t+\Delta t)\left[\mathcal{G}\ltwoinner{\varphi_1 ,\varphi_j}  + \dfrac{1}{z_F^2} \ltwoinner{D \varphi_1' ,\varphi_j'}- \dfrac{\mathcal{F}}{z_F} \ltwof{\varphi_1 ,\varphi_j'}\right. \nonumber\\
&&\left. - \frac{\mathcal{M}}{z_F}\ltwoinner{D \varphi_1 ,\varphi_j'} \right] \nonumber
\end{eqnarray}
is equivalent to the $(n-1) \times 1$  vector $T_ev_3(t)$ where \begin{eqnarray}
 v_3(t) &=& \rho^ {\text{atm}}(t+\Delta t)\left[\mathcal{G} \ltwoinner{\varphi_1 ,\varphi_2}  + \dfrac{1}{z_F^2} \ltwoinner{D \varphi_1' ,\varphi_2'}-\dfrac{\mathcal{F}}{z_F} \ltwof{ \varphi_1 ,\varphi_2'}- \dfrac{\mathcal{M}}{z_F}\ltwoinner{D \varphi_1 ,\varphi_2'} \right]e_1 \nonumber\\
& =&\rho^ {\text{atm}}(t+\Delta t)\, c_1 \, e_1 \, .\label{eq:v3}
 \end{eqnarray}
\end{itemize}
Standard inner products of the form $\ltwoinner{ \varphi_i, \varphi_j}$ and $\ltwoinner{ \varphi'_i, \varphi_j } $ are evaluated accurately in Appendix \ref{sec:sip}. As for the weighted inner products of the form
   $ \ltwoinner{ D \varphi'_i, \varphi_j }$, 
    $\ltwoinner{ D \varphi'_i, \varphi'_j }$, 
    $\ltwof{\varphi_i, \varphi_j}$ , 
    $\ltwof{ \varphi'_i, \varphi_j }$, 
they are approximated in Appendix \ref{sec:wip}. 
\noindent Consequently, all the vectors and matrices (except for $M$) are approximated as detailed in Appendices \ref{sec:matvec} and \ref{sec:matvec2} where for a uniform mesh with $h = h_{i+1} = h_{j+1} = z_{i+1}-z_{i}=z_{j+1}-z_{j}$ for all $i,j$ and $h_2 = z_2-z_1 = z_2=h$: 
\begin{itemize}
    \item $v_1(t) = (\rho^ {\text{atm}}(t+\Delta t)-\rho^ {\text{atm}}(t))\ltwof{\varphi_1,\varphi_2} e_1 
\approx (\rho^ {\text{atm}}(t+\Delta t)-\rho^ {\text{atm}}(t))\; \dfrac{f_1+f_2}{12} \;h_2\; e_1$
\item $
 v_3(t)   \approx  \rho^ {\text{atm}}(t+\Delta t)\left[\mathcal{G} \dfrac{h_2}{6}  - \dfrac{D_1+D_2}{2\,h_2\,z_F^2} -\mathcal{F}\,\dfrac{f_1+f_2}{4z_F}- \mathcal{M}\,\dfrac{D_1+D_2}{4\,{z_F}} \right]e_1 $, \qquad 
\item $
     A(D)\approx { \dfrac{1}{4}\begin{pmatrix}
            \eta_1 & \zeta_2 & 0 & \hdots & 0\vspace{3mm}\\
             -\zeta_2 & \eta_2 & \zeta_3 & & \vdots\vspace{3mm}\\
             0 & \ddots & \ddots & \ddots & 0\vspace{3mm}\\
              \vdots& & -\zeta_{n-2} & \eta_{n-2} & \zeta_{n-1}\vspace{3mm}\\
             0 & \hdots& 0 & -\zeta_{n-1} & \zeta_{n-1} \\
         \end{pmatrix}},
      $ with $\zeta_i = D_i+D_{i+1}$, and $\eta_i =  D_i-D_{i+2}$.
\item      $ S(D)\approx { \dfrac{1}{2h}\begin{pmatrix}
             {\zeta_1+\zeta_2}& -\zeta_2 & 0 & \hdots & 0\vspace{3mm}\\
             {-\zeta_2}& {\zeta_2+\zeta_3} & {-\zeta_3}& & \vdots\vspace{3mm}\\
             0 & \ddots & \ddots & \ddots & 0\vspace{3mm}\\
             \vdots & & {-\zeta_{n-2}} & \zeta_{n-2}+\zeta_{n-1} & {-\zeta_{n-1}}\vspace{3mm}\\
             0&\cdots&0&{-\zeta_{n-1}}&\zeta_{n-1}
         \end{pmatrix}},
$
with $\zeta_i = D_i+D_{i+1}$.
\item $
     K_f\approx { \dfrac{\mathcal{F}}{4}\begin{pmatrix}
            \eta_1 & \zeta_2 & 0 & \hdots & 0\vspace{3mm}\\
             -\zeta_2 & \eta_2 & \zeta_3 & & \vdots\vspace{3mm}\\
             0 & \ddots & \ddots & \ddots & 0\vspace{3mm}\\
              \vdots& & -\zeta_{n-2} & \eta_{n-2} & \zeta_{n-1}\vspace{3mm}\\
             0 & \hdots& 0 & -\zeta_{n-1} & \zeta_{n-1} \\
         \end{pmatrix}},
      $ with $\zeta_i = f_i+f_{i+1}$, and $\eta_i =  f_i-f_{i+2}$.
      \item      $ M_f\approx { \dfrac{h}{12}\begin{pmatrix}
             2{\zeta_1+2\zeta_2}& \zeta_2 & 0 & \hdots & 0\vspace{3mm}\\
             {\zeta_2}& {2\zeta_2+2\zeta_3} & {\zeta_3}& & \vdots\vspace{3mm}\\
             0 & \ddots & \ddots & \ddots & 0\vspace{3mm}\\
             \vdots & & {\zeta_{n-2}} & 2\zeta_{n-2}+2\zeta_{n-1} & {\zeta_{n-1}}\vspace{3mm}\\
             0&\cdots&0&{\zeta_{n-1}}&2\zeta_{n-1}
         \end{pmatrix}},
$
with $\zeta_i = f_i+f_{i+1}$.
 \item      $ M= { \dfrac{h}{6}\begin{pmatrix}
             4& 1 & 0 & \hdots & 0\\
             1& 4 & 1& & \vdots\\
             0 & \ddots & \ddots & \ddots & 0\\
             \vdots & & 1 &4 & 1\\
             0&\cdots&0&1&2
         \end{pmatrix}}.
$
      \end{itemize}

\subsection{Properties of Matrices}
Consider the discrete system
\begin{equation}\label{eq:Mat}
\big( M_f + T_e\Delta t C \big)\Lambda(t_{i+1})
=
M_f \Lambda(t_i) - T_e \Delta t\, b,
\qquad \Lambda(0)=0,
\end{equation}
where
 $ C = {\mathcal{G}} M+\dfrac{1}{z_F^2}S(D)-\dfrac{\mathcal{M}}{z_F}A(D)+\dfrac{1}{z_F}(B-K_f).$

\noindent We will first prove that most of the matrices are positive definite, specifically the $M_f$ matrix.
Then,  to prove the existence of a unique solution of system \eqref{eq:Mat}, we will  prove the invertibility of the matrix
\begin{eqnarray*}
V= M_f + T_e \Delta t C = M_f(I + T_e\Delta t M_f^{-1}C)
\end{eqnarray*}
Note that if $T_e\Delta t \,\rho (M_f^{-1}C) < 1$, then $V$ is invertible by Neumann series condition. A computable but more restrictive constraint is $T_e\Delta t  <  \dfrac{1}{\|M_f^{-1}\| \;\|C\|} \leq \dfrac{\|M_f\|}{\|C\|}$, since $\rho (M_f^{-1}C) \leq \|M_f^{-1}C \|\leq \|M_f^{-1}\|\;\|C \|$. 
However, we prove invertibility of $V$ by imposing a different constraints  on $\Delta t$.\\

\noindent Let
$y = (y_1,\dots,y_{n-1})^T \in \mathbb{R}^{n-1}, \quad y \neq 0,\quad
$
and define the nonzero finite element function
$$
v(z) := \sum_{i=1}^{n-1} y_i \varphi_{i+1}(z),
\qquad with \qquad 
(v(z))^2 = \left(\sum_{i=1}^{n-1} y_i \varphi_{i+1}(z)\right)^2
=
\sum_{i=1}^{n-1}\sum_{j=1}^{n-1}
y_i y_j \varphi_{i+1}(z)\varphi_{j+1}(z),
$$
where $\{\varphi_i\}$ are the piecewise linear basis functions, and $v(0)=0$.
\begin{lemma}\label{lemm:M}
The triadiagonal mass matrix $M$ is a symmetric positive definite
matrix.
\end{lemma}
\begin{proof}
$M$ is given by
$\displaystyle
M_{ij} = {\int_0^{1}}  \varphi_{i+1}(z)\varphi_{j+1}(z)\,dz = \big\langle \varphi_{i+1}(z), \,\varphi_{j+1}(z)   \big\rangle$.\vspace{2mm}
\noindent The symmetry of $M$ follows from the symmetry of the $L_2$ inner product; $M_{i,j}= M_{j,i}$. 
\noindent Then, for $y\neq 0$, i.e. $v(z) \neq 0$, 
$\displaystyle
y^T M y
= \sum_{i=1}^{n-1}\sum_{j=1}^{n-1}
y_i y_j M_{ij} = {\int_0^{1}} \sum_{i=1}^{n-1}\sum_{j=1}^{n-1}
y_i y_j  \varphi_{i+1}(z)\varphi_{j+1}(z)\,dz = 
{\int_0^{1}} v(z)^2\,dz > 0.
$
\end{proof}
\begin{lemma}\label{lemm:Mf}Given that $f(z) \geq 0$, then the tridiagonal weighted mass matrix $M_f$ is symmetric positive definite. Moreover, 
the approximated matrix $M_f$ is positive definite. 

\end{lemma}
\begin{proof} $M_f$ is given by
$\displaystyle
(M_f)_{ij} = \int_0^{z_F} f(z)\, \varphi_{i+1}(z)\varphi_{j+1}(z)\,dz = \ltwof{\,\varphi_{i+1}(z), \,\varphi_{j+1}(z)   }$.\vspace{2mm}

\noindent The symmetry of $M_f$ follows from the symmetry of the $L_f$ inner product; $(M_f)_{i,j}= (M_f)_{j,i}$.

\noindent Then, for $y\neq 0$ we have that $\displaystyle v(z) = \sum_{i=1}^{n-1} y_i \varphi_{i+1}(z) \neq 0$ over some subset of $(0,z_F)$, then\vspace{-2mm}
\begin{align*}
y^T M_f y
&=
\sum_{i=1}^{n-1}\sum_{j=1}^{n-1}
y_i y_j (M_f)_{ij} 
\;=
\sum_{i=1}^{n-1}\sum_{j=1}^{n-1}
y_i y_j
\int_0^{1} f(z)\varphi_{i+1}(z)\varphi_{j+1}(z)\,dz \\
&=
\int_0^{1} f(z)
\sum_{i=1}^{n-1}\sum_{j=1}^{n-1}
y_i y_j \varphi_{i+1}(z)\varphi_{j+1}(z)\,dz = \int_0^{1} f(z)\, v(z)^2\,dz
 \,=\,  \fnorm {v(z)}^2 > 0
\end{align*}
Note that $\fnorm{ v(z)} = 0$ if and only if $ v(z) = 0$ over $(0,z_F)$, which is not the case. \\

\noindent As for the approximated matrix $M_f$, it is also positive definite. We prove it for a uniform mesh:
\begin{eqnarray}
    y^T M_f y &=&\dfrac{h}{6}(f_{n-1}+f_n)y_{n-1}^2+ \dfrac{h}{6}\sum_{i=1}^{n-2}(f_i+2f_{i+1}+f_{i+2})
y_i^2 + \dfrac{h}{12}\sum_{i=1}^{n-2} 2y_iy_{i+1}(f_{i+1}+f_{i+2})\vspace{-1mm}\nonumber\\
&=&\dfrac{h}{6}\left[ \sum_{j=0}^{n-2}(f_{j+1}+f_{j+2})y_{j+1}^2 +  \sum_{i=1}^{n-2} y_iy_{i+1}(f_{i+1}+f_{i+2}) +  \sum_{i=1}^{n-2}(f_{i+1}+f_{i+2})
y_i^2\right]\vspace{-1mm}\nonumber\\
&=&\dfrac{h}{6}\left[ (f_{1}+f_{2})y_{1}^2 + \sum_{i=1}^{n-2}(f_{i+1}+f_{i+2})(y_{i+1}^2+y_iy_{i+1}+y_i^2)\right]\vspace{-1mm}\nonumber
 \end{eqnarray}
 \begin{eqnarray}
&=&\dfrac{h}{6}\left[ (f_{1}+f_{2})y_{1}^2 + \dfrac{1}{2}\sum_{i=1}^{n-2}(f_{i+1}+f_{i+2})\left({y_{i+1}^2+y_{i}^2} + {(y_{i+1}+y_{i})^2}\right)\right]\,>\,0\nonumber
\end{eqnarray}
{since $f_i = f(z_i)$ and $f$ is positive inside $(0,1)$ we have 
$f_i+f_{i+1}>0$ } 
and at least one $y_i\neq 0$. 
\end{proof}
\begin{lemma}\label{lemma:2.8}
Assuming $f(z)\geq 0$ is a strictly decreasing positive function, then the
tridiagonal matrix $K_f$ is positive definite and matrix $B$ is positive semi-definite.   
\end{lemma}
\begin{proof}
$K_f$ is given by
$\displaystyle
(K_f)_{ij} = \mathcal{F} \, \langle \varphi_{i+1}', \varphi_{j+1} \rangle_f =   \mathcal{F} \,\int_0^{1} f(z)\, \varphi_{i+1}'(z)\,\varphi_{j+1}(z)\,dz
$ where
\begin{align*}
\frac{4}{\mathcal{F}}y^T K_f y 
&= \frac{4}{\mathcal{F}}\sum_{i=1}^{n-1}\sum_{j=1}^{n-1} y_i\,y_j\, (K_f)_{ij} \,=\, 4 \int_0^{1} f(z)\, v'(z)\,v(z)\,dz = 2 \int_0^{1} f(z)\,(v(z)^2)'\,dz\\
&= 2 f(z_f) \, v(z_f)^2 - 2\int_0^{1} f'(z)\,v(z)^2\,dz \,\geq\, - 2\int_0^{1} f'(z)\,v(z)^2\,dz  \,>\, 0 
\end{align*}
since $v(0) = 0$, $f(z) \geq 0$ and $f'(z) < 0$. Moreover, the approximated matrix $K_f$ is also positive definite
\begin{align*}
\frac{4}{\mathcal{F}}y^T K_f y &= \dfrac{4}{\mathcal{F}}\sum_{i=1}^{n-1} y_i^2 (K_f)_{ii} + \dfrac{4}{\mathcal{F}}\sum_{i=1}^{n-2} y_i y_{i+1} \big( (K_f)_{i,i+1} + (K_f)_{i+1,i} \big) = \dfrac{4}{\mathcal{F}}\sum_{i=1}^{n-1} y_i^2 (K_f)_{ii} \\
&=   (f_{n-1} + f_n) y_{n-1}^2 + \sum_{i=1}^{n-2} (f_i - f_{i+2}) y_i^2 \,>\, 0
\end{align*}
\noindent since 
$
(K_f)_{i,i+1} + (K_f)_{i+1,i} 
= -\dfrac{\mathcal{F}}{4} (f_{i+1} + f_{i+2}) + \dfrac{\mathcal{F}}{4} (f_{i+1} + f_{i+2}) = 0
$, $f_n + f_{n-1}>0$, and $f(z)$ is strictly decreasing, i.e. $f_i - f_{i+2}>0$.  
\noindent Also,
$
\dfrac{1}{\mathcal{F}}y^T B y=f_n y_{n-1}^2\,\geq \, 0$ since $f(z)\geq 0$  and $y_{n-1}^2 \geq 0$.
\end{proof}

\begin{lemma}\label{lemm:SD}
Given that $D(z) \geq 0$, then the tridiagonal diffusion matrix $S(D)$ is symmetric. Assuming that $D(z) > 0$ is strictly positive over $[0,1)$ \eqref{A20}, then {$S(D)$ and the approximated matrix $S(D)$ are positive definite}.
\end{lemma}
\begin{proof}
$S(D)$ is given by
$\displaystyle
S_{ij}
=
\int_0^{1} D(z)\varphi_{i+1}'(z)\varphi_{j+1}'(z)\,dz =  \inner{ \varphi'_{i+1}(z), \varphi'_{j+1}(z)  }_D
$.\\
\noindent The symmetry of $S(D)$ follows from the symmetry of the $L_D$ inner product; $S_{i,j}= S_{j,i}$.

\noindent Then, \vspace{-6mm}
\begin{eqnarray}
  y^T S(D) y
&=&
\sum_{i=1}^{n-1}\sum_{j=1}^{n-1}
y_i y_j S_{ij} 
\;=
\sum_{i=1}^{n-1}\sum_{j=1}^{n-1}
y_i y_j
\int_0^{1} D(z)\varphi'_{i+1}(z)\varphi'_{j+1}(z)\,dz \nonumber \\
&=&
\int_0^{1} D(z)
\sum_{i=1}^{n-1}\sum_{j=1}^{n-1}
y_i y_j \varphi'_{i+1}(z)\varphi'_{j+1}(z)\,dz = \int_0^{1} D(z)\, v'(z)^2\,dz\nonumber \\
&=& \Dnorm{v'(z)}^2 \geq 0 \nonumber 
\end{eqnarray}
Note that $\Dnorm{ v'(z)}= 0$ if and only if $ v'(z) = 0$ for all $z\in (0,z_F)$. However, for all $y \neq 0$ we have
that $\displaystyle v'(z) =
\sum_{j=1}^{n-1}
y_j \varphi'_{j+1}(z)
 \neq 0$ since 
 $\varphi'_{j+1}(z) \neq 0$ for all $z\in (0,z_F)$ and their linear combination can never be equal to zero by definition of the $\mathbb{P}_1$ finite element basis functions.\\
 The approximated matrix $S(D)$ is also positive definite as shown in \cite{Moufawad2024Firn} where \\
$\displaystyle y^T S(D) y = \dfrac{1}{2h}\left[ (D_1+D_2)y_1^2 +  \sum_{i=2}^{n-1}(D_i+D_{i+1})(y_i - y_{i-1})^2\right] >0$ since $D_i+D_{i+1} >0$ for a strictly  positive function on $[0,1)$ (with $D_{n-1}+D_n \geq D_{n-1}>0$) and at least one of the terms $y_i-y_{i-1} \neq 0$ for $y\neq 0$.  Note that if $y$ is constant, then the first term is nonzero.  
\end{proof}
\begin{lemma}
Assuming $D(z)$ is a strictly decreasing positive function, then the
tridiagonal matrix $A(D)$ is positive definite.    
\end{lemma}
\begin{proof}
Since $\displaystyle 
A_{i,j} = \ltwoinner{ D\,\varphi_{i+1}',\,\varphi_{j+1} } = \int_0^{1} D(z)\, \varphi_{i+1}'(z)\,\varphi_{j+1}(z)\,dz $   then the proof is identical to Lemma \ref{lemma:2.8}.
\end{proof}
\noindent Now we are in a position to prove under what condition $V$ is positive definite, i.e. $y^tVy >0$ for all vectors $y\neq 0$ of length $n-1$, where for $G = \dfrac{1}{z_F}S(D) - \mathcal{M} A(D)$
    \begin{eqnarray*}
V&=& M_f + T_e\Delta t C = (M_f - \dfrac{T_e}{z_F}\Delta t K_f) + T_e\Delta t \mathcal{G} M+ \dfrac{T_e}{z_F}\Delta t \left(\dfrac{1}{z_F}S(D) - \mathcal{M} A(D) + B \right)\\
&=& (M_f - \dfrac{T_e}{z_F}\Delta t K_f) + T_e\Delta t \mathcal{G} M+ \dfrac{T_e}{z_F}\Delta t \left(G + B \right)
\end{eqnarray*}
then \vspace{-9mm}
\begin{eqnarray}
    y^tVy &=& (y^tM_fy - \dfrac{T_e}{z_F}\Delta t  y^tK_fy) + T_e\Delta t\mathcal{G} y^tMy + \dfrac{T_e}{z_F}\Delta t \left(  y^tGy + y^tBy\right)\nonumber \\
    &\geq &(y^tM_fy - \dfrac{T_e}{z_F}\Delta t  y^tK_fy) + T_e\Delta t\mathcal{G} \Mnorm{y}^2 + \dfrac{T_e}{z_F}\Delta t  y^tGy\label{eq:yvy}
    \end{eqnarray}
since $B$ is positive semi-definite.
We find lower bounds for the first and last terms as a function of the $h$-norm of vector $y$, $\hnorm{y} = \sqrt{h\norm{y}}$, leading then to lower bounds in terms of the $M$-norm,
$ \Mnorm{y} = \sqrt{y^tMy}$ since these norms are equivalent as proven in Lemma 4.1 in \cite{FirnA}.
\begin{remark}\label{rem:fh}
    For a fixed small $h>0$ , let $$ \underline{f}_{h} = \min\limits_{z\in [0,1-h]} f(z),$$ then
    \begin{itemize}
    \item $ \underline{f}_{h} > 0$ since $[0,1-h]\subset [0,1)$ and $f(z)>0$ over $[0,1)$ by \eqref{A10}.
        \item $ \underline{f}_{h} \geq f_{min} >0$,\;\; if $f(z_F)>0$
\end{itemize}    
\end{remark}
\begin{lemma}\label{lemma:m}
Assume that $f(z)\geq 0$ satisfies \eqref{A10} and that $\Delta t  < \dfrac{h}{6\mathcal{F}}\dfrac{z_F}{T_e}$, then  $$y^tM_fy - \dfrac{T_e}{z_F}\Delta t  y^tK_fy \;>\; \dfrac{\underline{f}_{h} }{24} \hnorm{y}^2 \;\geq\; \dfrac{\underline{f}_{h} }{24} \Mnorm{y}^2  $$
for $y\neq 0$ and 
 $ \underline{f}_{h} > 0$.
\end{lemma}
\begin{proof} Let $y\neq 0$, and $c=\dfrac{T_e}{z_F}\Delta t \dfrac{\mathcal{F}}{4}$. Then for $\Delta t  < \dfrac{h}{6\mathcal{F}}\dfrac{z_F}{T_e}$, $\implies c < \dfrac{h}{24} $ and  \\
\begin{eqnarray*}
y^tM_fy &-& \dfrac{T_e}{z_F}\Delta t  y^tK_fy\\
 &=& \dfrac{h}{6}\left[ (f_{1}+f_{2})y_{1}^2 + \dfrac{1}{2}\sum_{i=1}^{n-2}(f_{i+1}+f_{i+2})\left({y_{i+1}^2+y_{i}^2} + {(y_{i+1}+y_{i})^2}\right)\right] \\
    &&-c \left[(f_{n-1} + f_n) y_{n-1}^2 + \sum_{i=1}^{n-2} (f_i - f_{i+2}) y_i^2 \right]\\
    &=& \dfrac{2h}{12} (f_{1}+f_{2})y_{1}^2  -c(f_{n-1} + f_n) y_{n-1}^2+\dfrac{h}{12}\sum_{i=1}^{n-2}(f_{i+1}+f_{i+2})\left(y_{i}^2 + (y_{i+1}+y_{i})^2\right)\\
    &&+ \sum_{i=2}^{n-2}y_i^2\left[\dfrac{h}{12}(f_{i}+f_{i+1})-c (f_{i}-f_{i+2})\right]  + \dfrac{h}{12}(f_{n-1} + f_n) y_{n-1}^2-c(f_1-f_3)y_1^2
      \end{eqnarray*}
    \begin{eqnarray*}
  &=& \dfrac{h}{12} (f_{1}+f_{2})y_{1}^2 +\dfrac{h}{12}\sum_{i=1}^{n-2}(f_{i+1}+f_{i+2})\left(y_{i}^2 + (y_{i+1}+y_{i})^2\right)\\
    &&+ \sum_{i=1}^{n-2}y_i^2\left[\left(\dfrac{h}{12} -c\right)f_{i}+\dfrac{h}{12}f_{i+1}+c f_{i+2}\right]  + (\dfrac{h}{12}-c)(f_{n-1} + f_n) y_{n-1}^2\\
&\geq & \sum_{i=1}^{n-1}y_i^2\left(\dfrac{h}{12} -c\right)f_{i}\;>\; \dfrac{ \underline{f}_{h}}{24}  \sum_{i=1}^{n-1} hy_{i}^2 = \dfrac{ \underline{f}_{h}}{24}  \hnorm{y}^2 \;\geq\; \dfrac{ \underline{f}_{h}}{24} \Mnorm{y}^2
\end{eqnarray*} 
  $f(z)$ a strictly  positive function over $[0,1)$, thus $f_1, f_2, \cdots , f_{n-1} > 0$ since $z_1, z_2, \cdots, z_{n-1} \in [0,1-h] $  with $0<  \underline{f}_{h} \leq f(z), \;\;  \forall z \in [0,1-h] \subset [0,1) $ by Remark \ref{rem:fh}.
\end{proof}
\begin{lemma}\label{lemma:g}
Assume that $D(z)\geq 0$ \eqref{A20} is a Lipschitz continuous function \eqref{A2pp0},  then  there exists an $h_0<1$ ($n_0>1$), and there exist a constant $K_G$,  such that for $h<h_0$ ($n>n_0$)    one has
$$v^TGv\ge -|K_G|. \Mnorm{v}^2.$$
where $K_G =   \dfrac{1}{2\,c_D\,I(D)}-\dfrac{\mathcal{M}}{2}L_\delta$, with $0<c_D<2$ independent of $h$, $\displaystyle I(D) = z_F \int_0^1\dfrac{1}{D(z)}dz < \infty$, and $L_\delta$ the Lipschitz continuity constant such that $
    |D(z)-D(y)|\le L_\delta |z-y|,\quad\forall z,\,y\in[0,1-\delta].$  
\end{lemma}
\begin{proof}
    Refer to Lemma 4.5 in \cite{FirnA}.
\end{proof}
\begin{theorem}
Assume that $D(z)\geq 0$ \eqref{A20} is a Lipschitz continuous function \eqref{A2pp0} and that $f(z)\geq 0$ satisfies \eqref{A10}. Let $h$ be chosen to be sufficiently small ($h<h_0$), then the matrix of the discrete system $$V = M_f+T_e\,\Delta t\, C $$ is invertible  if \vspace{-2mm}
$$0 < \Delta t <\dfrac{z_F}{6T_e} min\left\{\dfrac{h}{\mathcal{F}}, \dfrac{  \underline{f}_{h}}{4\left|z_F\mathcal{G} - |K_G|\,\right|  }\right\}$$ where
 $ \underline{f}_{h} = \min\limits_{z\in [0,1-h]} f(z) >0$.
    \end{theorem}
    \begin{proof}
Using the result of the lemmas \ref{lemma:m} and \ref{lemma:g} for  $\Delta t  < \dfrac{h}{6\mathcal{F}}\dfrac{z_F}{T_e}$ in \eqref{eq:yvy}, we get  
\begin{eqnarray*}
    y^tVy 
    &\geq &(y^tM_fy - \dfrac{T_e}{z_F}\Delta t  y^tK_fy) + T_e\Delta t\mathcal{G} \Mnorm{y}^2 + \dfrac{T_e}{z_F}\Delta t  y^tGy\\
    &\geq &\left(\dfrac{ \underline{f}_{h}}{24} + T_e\Delta t\mathcal{G} - |K_G| \dfrac{T_e}{z_F}\Delta t\right) \Mnorm{y}^2 = \left(\dfrac{ \underline{f}_{h}}{24} + \dfrac{T_e}{z_F}\Delta t (z_F\mathcal{G} - |K_G| )\right) \Mnorm{y}^2 
    \end{eqnarray*}
    If $z_F\mathcal{G} - |K_G| \geq 0$, then $y^tVy >0$ for all $y\neq 0$.\\
    If $z_F\mathcal{G} - |K_G| < 0$, then $y^tVy >0$ for all $y\neq 0$ if and only if 
    \begin{eqnarray*}
        &&\dfrac{\underline{f}_{h}}{24} + \dfrac{T_e}{z_F}\Delta t (z_F\mathcal{G} - |K_G| ) \;>\;0\\
        \iff &&\Delta t (z_F\mathcal{G} - |K_G| ) \;>\;-\dfrac{z_F}{T_e}\dfrac{\underline{f}_{h}}{24}\\
         \iff &&\Delta t \;<\;\dfrac{z_F\,\underline{f}_{h}}{24T_e(z_F\mathcal{G} - |K_G| ) }
    \end{eqnarray*}
    Then, for $\Delta t <\dfrac{z_F}{6T_e} min\left\{\dfrac{h}{\mathcal{F}}, \dfrac{ \underline{f}_{h}}{4|z_F\mathcal{G} - |K_G|\,|  }\right\}$, $y^tVy >0$ for all $y\neq 0$ implying that $V$ is positive definite, i.e. invertible.
    \end{proof}
        
\section{Conclusion}
In this paper, we consider the model developed in \cite{Witrant2012acp}, where both the diffusion coefficient $D(z)$ and the average volume fraction in the open pores $f(z)$ are positive continuous functions that may degenerate at $z_F$  in a  similar behavior. We derive the variational formulation and prove the existence and uniqueness of the solution using Lion's theorem and a Hardy-type inequality. 

We discretize the rescaled semi-variational formulation using Euler-Implicit time scheme and P1 Finite element space scheme, put it in matrix form, and study the properties of the obtained matrices.
 By imposing sufficient conditions on $\Delta t$, we prove the existence of a unique solution for the discrete system.

 As a future work, besides testing the robustness of the discrete scheme and solving the corresponding inverse, we will consider the change of variable $u=e^{\kappa t} u$ that may allow us to prove the invertibility of the discrete system $V$ without any constraint on $\Delta t$ in terms of $h$.
      
\addcontentsline{toc}{section}{References}
\bibliographystyle{plain}
\bibliography{references}
\newpage
\begin{appendix}
\section{$\ltwof {v, w}$ is an inner product on \(L_f\) Under Assumption~\ref{A10}}
\label{sec:app}
To prove that $\displaystyle \ltwof{v,w} = \int_0^{z_F} f(z) v(z) w(z)\, dz$ , defines an inner product in \( L_f(0, z_F) \), we need to verify that it satisfies  symmetry, bilinearity, and positivity for all $u,v,w \in  L_f(0, z_F)$.
\begin{enumerate}
\item \textbf{Symmetry:} $ \displaystyle \ltwof{ v, w } = \int_0^{z_F} f(z) v(z) w(z)\, dz = \int_0^{z_F} f(z) w(z) v(z)\, dz =\ltwof{ w, v }$ 
\item \textbf{Bilinearity:} For all $a,b \in \mathbb{R}$
\begin{eqnarray*}
\ltwof{ au+bv, w } &=&  \int_0^{z_F} f(z) (au(z)+bv(z)) w(z)\, dz\\
&=& a \int_0^{z_F} f(z) u(z)w(z)\, dz   + b\int_0^{z_F} f(z) v(z) w(z)\, dz   \\
&=&a\ltwof{ u, w }+b\ltwof{ v, w }
\end{eqnarray*}
    \item \textbf{Positivity:} since $f(z) \geq 0 $ then for $v(z) \neq 0$\\
    $$\displaystyle \ltwof{ v, v } = \int_0^{z_F} f(z) v^2(z) \, dz\geq 0.$$
    Moreover, if $ \ltwof{  v, v } = 0$ this implies that $v(z) = 0$ since 
    $ f(z) > 0 $ for all $ z \in [0, z_F) $, and for all $b<z_F$ we have that:\begin{eqnarray}
        \displaystyle \ltwof{v, v } = \int_0^{b} f(z) v^2(z) + \int_b^{z_F} f(z) v^2(z) \, dz \;=\; 0 \label{eq:vv}
    \end{eqnarray}
    where $ \displaystyle\int_0^{b} f(z) v^2(z) =0$. 
     Thus, $v(z) = 0$ in $L_f(0,b)$. And since $(0,z_F) = (0,b]\cup [b,z_F)$ for all $b<z_F.$ Then by letting $b$ tend to $ z_F$ we get that $v(z) = 0$ in $L_f(0,z_F)$, as the single point $z_F$ is a set of measure zero.
\end{enumerate}

\section{\texorpdfstring{\(\ltwoDf {v, w}\) is an inner product on \(H^1_{D,f}\) Under Assumptions~\ref{A10} and \ref{A20}}{\(\ltwof {v, w}\) is an inner product on \(L^2_f\) Under Assumption~\ref{A1}}}\label{sec:appDf}

To prove that $\displaystyle \ltwoDf{v,w} = \ltwof{v,w}+\ltwoD{v_z,w_z}$ , defines an inner product in \( H^1_{D,f}(0, z_F) \), we note that $\ltwoD{v,w}$ is an inner product over $L_D^2(0,z_F)$ under assumption \eqref{A20}.  Accordingly,  $\forall u,v,w \in  H^1_{D,f}(0, z_F)$, then $u_z, v_z, w_z \in L_D^2(0,z_F)$.
\begin{enumerate}
\item \textbf{Symmetry:} $ \displaystyle \ltwoDf{ v, w } = \ltwof{v,w}+\ltwoD{v_z,w_z}= \ltwof{w,v}+\ltwoD{w_z, v_z} = \ltwoDf{ w, v }$ 
\item \textbf{Bilinearity:} For all $a,b \in \mathbb{R}$ and by bilinearity of $\ltwof{.,.}$ and $\ltwoD{.,.}$
\begin{eqnarray*}
\ltwoDf{ au+bv, w } &=&  \ltwof{ au+bv, w }+ \ltwoD{ au_z+bv_z, w_z }\\
&=& a 
\ltwof{ u, w } +  b\ltwof{ v, w } + a
\ltwoD{ u_z, w_z } + b\ltwoD{ v_z, w_z }   \\
&=&a\ltwoDf{ u, w }+b\ltwoDf{ v, w }
\end{eqnarray*}
    \item \textbf{Positivity:} For $v(z) \neq 0$, we have that \\
    $$\displaystyle \ltwoDf{ v, v }=\ltwof{v,v}+\ltwoD{v_z,v_z} = \fnorm{v}^2+\Dnorm{v_z}^2 \geq  \fnorm{v}^2 \geq 0.$$
    Moreover, if $ \ltwoDf{  v, v } = 0$ this implies that $v(z) = 0$ by positivity of $L_f^2$ inner product since
  $$\displaystyle  0\leq \fnorm{v}^2 \leq  \ltwoDf{ v, v }  =  0\qquad \implies \fnorm{v}^2 = 0 \qquad \implies {v} = 0. $$
\end{enumerate}

\section{Matrix Assembly}
To assemble the matrices and vectors appearing in the matrix form \eqref{eq:mat2} of the discrete system \eqref{Firn-tau33},
\noindent we need to compute the following inner products  for \(i, j = 1, 2, \dots, n\).\\ Note that for the rescaled problem $\mathbf {z_F = 1}$ in the below inner products since we are working in $L^2(0,1)$, $L^2_f(0,1)$, $L^2_D(0,1), \mbox{ and }H^1_{D,f}(0,1)$.

\begin{itemize}
    \item Standard inner products: $$\displaystyle
    \langle \varphi_i, \varphi_j \rangle = \int_0^{z_F} \varphi_i(z) \varphi_j(z) \, dz,\qquad 
       \langle \varphi'_i, \varphi_j \rangle = \int_0^{z_F} \varphi'_i(z) \varphi_j(z) \, dz
  $$    
    \item Weighted inner products with weight function  \(D(z)\):
    \begin{eqnarray*}
    \langle D \varphi'_i, \varphi_j \rangle = \int_0^{z_F} D(z) \varphi_i'(z) \varphi_j(z) \, dz, 
    \quad
    \langle D \varphi'_i, \varphi'_j \rangle = \int_0^{z_F} D(z) \varphi_i'(z) \varphi_j'(z) \, dz,
    \quad
    \end{eqnarray*}
    \item Weighted inner products with weight function \( f(z) \):
    \begin{eqnarray*}
    \langle \varphi_i, \varphi_j \rangle_f = \int_0^{z_F} f(z) \varphi_i(z) \varphi_j(z) \, dz,
    \quad
    \langle \varphi'_i, \varphi_j \rangle_f = \int_0^{z_F} f(z) \varphi'_i(z) \varphi_j(z) \, dz
    \end{eqnarray*}
\end{itemize}
First, we consider standard \( L^2 \) inner products, then the weighted inner products, to finally assemble the matrices.
For that purpose, note that,\\
\\
\noindent
$\displaystyle
\varphi_1(z) =
\begin{cases}
\dfrac{z_2 - z}{z_2 - z_1}, & z_1 \le z \le z_2,\\[6pt]
0, & \text{otherwise},
\end{cases}
\qquad
\varphi_n(z) =
\begin{cases}
\dfrac{z - z_{n-1}}{z_n - z_{n-1}}, & z_{n-1} \le z \le z_n,\\[6pt]
0, & \text{otherwise},
\end{cases}
$

\medskip

\noindent
$\displaystyle
\varphi_i(z) =
\begin{cases}
\dfrac{z - z_{i-1}}{z_i - z_{i-1}}, & z_{i-1} \le z \le z_i,\\[6pt]
\dfrac{z_{i+1} - z}{z_{i+1} - z_i}, & z_i \le z \le z_{i+1},\\[6pt]
0, & \text{otherwise},
\end{cases}
\qquad i=2,\dots,n-1.
$

\bigskip

\noindent
$\displaystyle
\varphi_1'(z) =
\begin{cases}
-\dfrac{1}{z_2 - z_1}, & z_1 \le z \le z_2,\\[6pt]
0, & \text{otherwise},
\end{cases}
\qquad
\varphi_n'(z) =
\begin{cases}
\dfrac{1}{z_n - z_{n-1}}, & z_{n-1} \le z \le z_n,\\[6pt]
0, & \text{otherwise},
\end{cases}
$

\medskip

\noindent
$\displaystyle
\varphi_i'(z) =
\begin{cases}
\dfrac{1}{z_i - z_{i-1}}, & z_{i-1} \le z \le z_i,\\[6pt]
-\dfrac{1}{z_{i+1} - z_i}, & z_i \le z \le z_{i+1},\\[6pt]
0, & \text{otherwise},
\end{cases}
\qquad i=2,\dots,n-1.
$\\\\
 where we consider a nonuniform mesh with nodes \( z_1 < z_2 < \cdots < z_n \) and $h_i = z_i-z_{i-1}$
  denotes the length of the $(i-1)^{th}$ mesh element. If $h = h_i = h_j$ for all $i,j$ then the meshing is uniform. 
 The weight function values at nodes are denoted by $ f_i = f(z_i) $ and $D_i = D(z_i)$.

\subsection{Evaluating Standard Inner Products}\label{sec:sip}

We begin by computing the standard inner products $\displaystyle \langle \varphi_i, \varphi_j \rangle = \int_{0}^{z_F} \varphi_i \varphi_j \;dz $ for $i,j = 1,2,...,n$.
\begin{itemize}
\item For \( j = i-1 \) and \( i \neq 1 \) :
\begin{eqnarray}
\langle \varphi_i, \varphi_{i-1} \rangle &=& \int_{z_{i-1}}^{z_i} \varphi_i\varphi_{i-1} \, dz = \int_{z_{i-1}}^{z_i} \frac{z - z_{i-1}}{z_i - z_{i-1}} \cdot \frac{z_i - z}{z_i - z_{i-1}} \, dz \nonumber \\
&=& \frac{1}{(z_i - z_{i-1})^2} \int_{z_{i-1}}^{z_i} (z z_i - z^2 - z_i z_{i-1} + z z_{i-1}) \, dz \nonumber \\
&=& \frac{1}{(z_i - z_{i-1})^2} \left( \frac{z^2 z_i}{2} - \frac{z^3}{3} - z_i z_{i-1} z + \frac{z^2 z_{i-1}}{2} \right) \Bigg|_{z_{i-1}}^{z_i} 
\,=\, \frac{z_i - z_{i-1}}{6}\,=\, \frac{h_i}{6}\nonumber 
\end{eqnarray}
\item For \( j = i+1 \) and \( i \neq n \):
\begin{eqnarray}
\langle \varphi_i, \varphi_{i+1} \rangle &=& \int_{z_i}^{z_{i+1}} \varphi_i \varphi_{i+1} \, dz = \int_{z_i}^{z_{i+1}} \frac{z_{i+1} - z}{z_{i+1} - z_i} \cdot \frac{z - z_i}{z_{i+1} - z_i} \, dz \nonumber \\
&=& \frac{1}{(z_{i+1} - z_i)^2} \int_{z_i}^{z_{i+1}} (z z_{i+1} - z_{i+1} z_i - z^2 + z z_i) \, dz 
\;=\; \frac{z_{i+1} - z_i}{6}\;=\; \frac{h_{i+1} }{6} \nonumber 
\end{eqnarray}
\item For \( j = i = n \):
\begin{eqnarray}
\langle \varphi_n, \varphi_n \rangle &=& \int_{z_{n-1}}^{z_n} \varphi_n^2 \, dz = \int_{z_{n-1}}^{z_n} \left( \frac{z - z_{n-1}}{z_n - z_{n-1}} \right)^2 \, dz 
\,=\, \frac{z_n - z_{n-1}}{3}\,=\, \frac{h_n }{3}\nonumber
\end{eqnarray}
\item For \( j = i = 1 \):
\begin{eqnarray}
\langle \varphi_1, \varphi_1 \rangle
&=& \int_{z_1}^{z_2} \varphi_1^2 \, dz
= \int_{z_1}^{z_2} \left( \frac{z_2 - z}{z_2 - z_1} \right)^2 dz 
\,=\, \frac{1}{(z_2 - z_1)^2}
\left[ \frac{(z_2 - z)^3}{-3} \right]_{z_1}^{z_2} 
\,=\, \frac{z_2 - z_1}{3}\,=\, \frac{h_2 }{3}\nonumber
\end{eqnarray}

\item For \( j = i \neq n \), \( j = i \neq 1 \) :
\begin{eqnarray}
\langle \varphi_i, \varphi_i \rangle &=& \int_{z_{i-1}}^{z_i} \varphi_i^2 \, dz + \int_{z_i}^{z_{i+1}} \varphi_i^2 \, dz 
\,=\, \frac{z_i - z_{i-1}}{3} + \frac{z_{i+1} - z_i}{3} = \frac{z_{i+1} - z_{i-1}}{3}=\frac{h_{i} + h_{i+1}}{3}\nonumber
\end{eqnarray}
\end{itemize}
So,
\begin{eqnarray}
\langle \varphi_i, \varphi_j \rangle = 
\dfrac{1}{6}\begin{cases}
2 h_2 
& \text{if } i = j = 1,\vspace{2mm}\\
2 ({h_{i+1} + h_{i}});
& \text{if }$$ i = j \neq n ,  i = j \neq 1 $$ \vspace{2mm} \\ 
2 h_n 
& \text{if } i = j = n, \vspace{2mm}\\
h_i 
& \text{if } j = i-1,  i\neq 1\vspace{2mm}\\
h_{i+1}
& \text{if } j = i+1,  i\neq n\vspace{2mm}\\
0, & \text{otherwise}.
\end{cases} \qquad
\label{inner}
\end{eqnarray}\\
\noindent As for the standard inner product $ \displaystyle \langle \varphi'_i, \varphi_j \rangle = \int_{0}^{z_F} \varphi_i' \varphi_j \;dz$, we get:
\begin{itemize}
\item For \( j = i - 1 \) and \( i \neq 1 \):
\begin{eqnarray}
\langle \varphi'_i, \varphi_{i-1} \rangle &=& \int_{z_{i-1}}^{z_i} \varphi'_i \, \varphi_{i-1} \, dz = \int_{z_{i-1}}^{z_i} \frac{1}{z_i - z_{i-1}} \cdot \frac{z_i - z}{z_i - z_{i-1}} \, dz \nonumber \\
&=& \frac{1}{(z_i - z_{i-1})^2} \left[- \frac{1}{2} (z_i - z)^2 \right]_{z_{i-1}}^{z_i} = \frac{1}{2}\nonumber 
\end{eqnarray}
\item For \( j = i + 1 \) and \( i \neq n \):
\begin{eqnarray}
\langle \varphi'_i, \varphi_{i+1} \rangle &=& \int_{z_i}^{z_{i+1}} \varphi'_i \, \varphi_{i+1} \, dz = \int_{z_i}^{z_{i+1}} \frac{-1}{z_{i+1} - z_i} \cdot \frac{z - z_i}{z_{i+1} - z_i} \, dz \nonumber \\
&=& -\frac{1}{(z_{i+1} - z_i)^2} \left[ \frac{1}{2} (z - z_i)^2 \right]_{z_i}^{z_{i+1}} = -\frac{1}{2}\nonumber 
\end{eqnarray}

\item For \( j = i \neq n \), \( j = i \neq 1 \) :
\begin{eqnarray}
\langle \varphi'_i, \varphi_i \rangle &=& \int_{z_{i-1}}^{z_i} \varphi'_i \, \varphi_i \, dz + \int_{z_i}^{z_{i+1}} \varphi'_i \, \varphi_i \, dz \nonumber \\
&=& \int_{z_{i-1}}^{z_i} \frac{1}{z_i - z_{i-1}} \cdot \frac{z - z_{i-1}}{z_i - z_{i-1}} \, dz + \int_{z_i}^{z_{i+1}} \frac{-1}{z_{i+1} - z_i} \cdot \frac{z_{i+1} - z}{z_{i+1} - z_i} \, dz \nonumber \\
&=& \frac{1}{(z_i - z_{i-1})^2}
\left[ \frac{1}{2}(z - z_{i-1})^2 \right]_{z_{i-1}}^{z_i}
- \frac{1}{(z_{i+1} - z_i)^2}
\left[ \frac{-1}{2}(z_{i+1} - z)^2 \right]_{z_i}^{z_{i+1}} \nonumber \\
&=& \frac{1}{(z_i - z_{i-1})^2}
\left[ \frac{1}{2}(z_i - z_{i-1})^2 \right]
- \frac{1}{(z_{i+1} - z_i)^2}
\left[ \frac{1}{2}(z_{i+1} - z_i)^2 \right] 
= \frac{1}{2} - \frac{1}{2} = 0 \nonumber
\end{eqnarray}
\item For \( j = i = n \):
\begin{eqnarray}
\langle \varphi'_n, \varphi_n \rangle &=& \int_{z_{n-1}}^{z_n} \varphi'_n \, \varphi_n \, dz = \int_{z_{n-1}}^{z_n} \frac{1}{z_n - z_{n-1}} \cdot \frac{z - z_{n-1}}{z_n - z_{n-1}} \, dz \nonumber \\
&=& \frac{1}{(z_n - z_{n-1})^2} \left[ \frac{1}{2} (z - z_{n-1})^2 \right]_{z_{n-1}}^{z_n} = \frac{1}{2}\nonumber
\end{eqnarray}
\item For \( j = i = 1 \):
\begin{eqnarray}
\langle \varphi'_1, \varphi_1 \rangle &=& \int_{z_1}^{z_2} \varphi'_1 \, \varphi_1 \, dz
= \int_{z_1}^{z_2} \frac{-1}{z_2 - z_1} \cdot \frac{z_2 - z}{z_2 - z_1} \, dz 
= - \frac{1}{(z_2 - z_1)^2} \int_{z_1}^{z_2} (z_2 - z) \, dz \nonumber \\
&=& - \frac{1}{(z_2 - z_1)^2} \left[ -\frac{1}{2} (z_2 - z)^2 \right]_{z_1}^{z_2} 
= - \frac{1}{(z_2 - z_1)^2} \cdot \left(  \frac{(z_2 - z_1)^2}{2} \right) 
= -\frac{1}{2}\nonumber
\end{eqnarray}
\end{itemize}
So, in summary: 
\begin{eqnarray}
\langle \varphi'_i, \varphi_j \rangle = 
\dfrac{1}{2}\,\begin{cases}
-1, & \text{if } i = j = 1, \vspace{2mm} \\
0, & \text{if } i = j \neq n, \, i =j \neq 1, \vspace{2mm} \\
1, & \text{if } i = j = n, \vspace{2mm} \\
1, & \text{if } j = i-1, \; i \neq 1, \vspace{2mm} \\
-1, & \text{if } j = i+1, \; i \neq n, \vspace{2mm} \\
0, & \text{otherwise}.
\end{cases} \qquad
\label{inner_prime}
\end{eqnarray}

\noindent 
\subsection{Approximating Weighted Inner Products}\label{sec:wip}
To approximate the weighted inner products, we use the Mean Value Theorem (MVT), that states that for two continuous functions $f(x)$ and $g(x)$, there exists a constant $c \in (a,b)$ such that:
\begin{eqnarray}
\int_a^b f(x)g(x)\,dx &=& f(c) \int_a^b g(x)\,dx , 
\end{eqnarray}
Moreover, we can approximate $f(c)$ to be $f(c) \approx \frac{1}{2}[f(a) + f(b)]$ to get
\begin{equation}\label{eq:MVT}
    \int_a^b f(x)g(x)\,dx \approx \dfrac{f(a) + f(b)}{2} \int_a^b g(x)\,dx , 
\end{equation}

\subsubsection{$D$ Inner Products} \label{sec:dip}
In this section, we will approximate the $D$- weighted inner products $(D \varphi_i', \varphi_j)$ and  $(D \varphi_i', \varphi_j')$.\\
For $\displaystyle
\langle D\varphi_i', \varphi_j \rangle = \int_0^{z_F} D\, \varphi_i' \varphi_j \,dz
 $, using MVT \eqref{eq:MVT} and \eqref{inner_prime} we get:

\begin{itemize}
\item for $j = i - 1$ and $i \neq 1$:
\begin{eqnarray}
\langle D \varphi_i', \varphi_{i-1} \rangle &=& \int_{z_{i-1}}^{z_i} D \varphi_i' \varphi_{i-1} \,dz 
\,\approx\, \frac{D(z_{i-1}) + D(z_i)}{2}\langle \varphi_i', \varphi_{i-1} \rangle 
\;=\;\frac{D(z_{i-1}) + D(z_i)}{4}\nonumber
\end{eqnarray}
\item for $j = i + 1$ and $i \neq n$:
\begin{eqnarray}
\hspace{-3mm}\langle D\varphi_i', \varphi_{i+1} \rangle &=& \int_{z_i}^{z_{i+1}} D \varphi_i' \varphi_{i+1} \,dz
\,\approx\, \frac{D(z_i) + D(z_{i+1})}{2}\langle \varphi_i', \varphi_{i+1} \rangle 
\;=\; -\frac{D(z_i) + D(z_{i+1})}{4}\nonumber
\end{eqnarray}
\item for $j = i \neq n, j =i \neq 1$:
\begin{eqnarray}
\langle D \varphi_i', \varphi_i \rangle &=& \int_{z_{i-1}}^{z_i} D \varphi_i' \varphi_i \,dz + \int_{z_i}^{z_{i+1}} D \varphi_i' \varphi_i \,dz \nonumber \\
&\approx& \frac{D(z_{i-1}) + D(z_i)}{2}\langle \varphi_i', \varphi_i \rangle_{[z_{i-1},z_i]} 
+ \frac{D(z_i) + D(z_{i+1})}{2}\langle \varphi_i', \varphi_i \rangle_{[z_i,z_{i+1}]} \nonumber \\
&=& \frac{D(z_{i-1}) + D(z_i)}{4} - \frac{D(z_i) + D(z_{i+1})}{4}
\,=\, \frac{D(z_{i-1}) - D(z_{i+1})}{4}\nonumber
\end{eqnarray}

\item for $j = i = n$:
\begin{eqnarray}
\langle D \varphi_n', \varphi_n \rangle &=& \int_{z_{n-1}}^{z_n} D \varphi_n' \varphi_n \,dz 
\,\approx\, \frac{D(z_{n-1}) + D(z_n)}{2} \langle \varphi_n', \varphi_n \rangle 
\;=\; \frac{D(z_{n-1}) + D(z_n)}{4}\nonumber 
\end{eqnarray}
\item for $j = i = 1$:
\begin{eqnarray}
\langle D \varphi_1', \varphi_1 \rangle
&=& \int_{z_1}^{z_2} D \, \varphi_1' \, \varphi_1 \, dz
\,\approx\, \frac{ D(z_1) + D(z_2)}{2}
\langle \varphi_1', \varphi_1 \rangle  
= -\frac{D(z_1) + D(z_2) }{4} \nonumber
\end{eqnarray}

\end{itemize}
Hence, we get the following approximation for $i, j = 1, 2, \ldots, n$:
\begin{eqnarray}
\langle D \varphi'_i, \varphi_j \rangle \approx
\dfrac{1}{4}\,\begin{cases}
- (D(z_1) + D(z_2) ),
& \text{if } i=j=1 \vspace{2mm} \\

 D(z_{i-1}) - D(z_{i+1}),
& \text{if } i=j\neq n,\; i=j\neq 1 \vspace{2mm} \\

 D(z_{n-1}) + D(z_n) ,
& \text{if } i=j=n \vspace{2mm} \\

 D(z_{i-1}) + D(z_i) ,
& \text{if } j=i-1,\; i\neq 1 \vspace{2mm} \\

- ( D(z_i) + D(z_{i+1})),
& \text{if } j=i+1,\; i\neq n \vspace{2mm} \\

0, & \text{otherwise}.
\end{cases}
\qquad
\label{inner_D_prime}
\end{eqnarray}

\noindent Similarly for $\displaystyle \langle D \varphi'_i, \varphi_j' \rangle = \int_0^{z_F} D\varphi_i' \varphi_j'  dz$, using MVT \eqref{eq:MVT}
we get:
\begin{itemize}
\item for $j = i - 1$ and $i \neq 1$ :
\begin{eqnarray}
 \langle D \varphi_i', \varphi_{i-1}' \rangle &=& \int_{z_{i-1}}^{z_i} D \varphi_i' \varphi_{i-1}' dz 
\,\approx\, -\frac{D(z_{i-1}) + D(z_i)}{2} \frac{z_i - z_{i-1}}{(z_i - z_{i-1})^2} 
= -\frac{D(z_{i-1}) + D(z_i)}{2(z_i - z_{i-1})} \nonumber
\end{eqnarray}

\item for $j = i + 1$ and $i \neq n$:
\begin{eqnarray}
 \langle D \varphi_i', \varphi_{i+1}' \rangle &=& \int_{z_i}^{z_{i+1}} D \varphi_i' \varphi_{i+1}' dz 
\,\approx\, -\frac{D(z_i) + D(z_{i+1})}{2} \frac{z_{i+1} - z_i}{(z_{i+1} - z_i)^2}
= -\frac{D(z_i) + D(z_{i+1})}{2(z_{i+1} - z_i)}\nonumber
\end{eqnarray}

\item for $j = i \neq n, j=i \neq 1$:
\begin{eqnarray}
\langle D \varphi_i', \varphi_i' \rangle &=& \int_{z_{i-1}}^{z_i} D \varphi_i'^2 dz + \int_{z_i}^{z_{i+1}} D \varphi_i'^2 dz \nonumber \\
&\approx& \frac{D(z_{i-1}) + D(z_i)}{2}\frac{z_i - z_{i-1}}{(z_i - z_{i-1})^2} + \frac{D(z_i) + D(z_{i+1})}{2} \frac{z_{i+1} - z_i}{(z_{i+1} - z_i)^2} \nonumber \\
&=& \frac{D(z_{i-1}) + D(z_i)}{2(z_i - z_{i-1})} + \frac{D(z_i) + D(z_{i+1})}{2(z_{i+1} - z_i)}\nonumber
\end{eqnarray}

\item for $j = i = n$:
\begin{eqnarray}
 \langle D \varphi_n', \varphi_n' \rangle &=& \int_{z_{n-1}}^{z_n} D \varphi_n'^2 dz 
\,\approx\, \frac{D(z_{n-1}) + D(z_n)}{2}\frac{(z_n - z_{n-1})}{(z_n - z_{n-1})^2} 
\,=\, \frac{D(z_{n-1}) + D(z_n)}{2(z_n - z_{n-1})} \nonumber 
\end{eqnarray}
\item for $j = i = 1$:
\begin{eqnarray}
\langle D \varphi_1', \varphi_1' \rangle
&=& \int_{z_1}^{z_2} D \, \varphi_1'^2 \, dz 
\,\approx\, \frac{D(z_1) + D(z_2)}{2}
\frac{(z_2 - z_1) }{(z_2 - z_1)^2}
\,=\,\frac{D(z_1) + D(z_2)}{2(z_2 - z_1)}\nonumber 
\end{eqnarray}

\end{itemize}

Hence,we get the following approximation for $i, j = 1, 2, \ldots, n$:
\begin{eqnarray}
\langle D \varphi_i', \varphi_j' \rangle \approx \dfrac{1}{2}
\begin{cases}
\dfrac{D(z_1) + D(z_2)}{h_2},
& \text{if } i=j=1 \vspace{2mm} \\

\dfrac{D(z_{i-1}) + D(z_i)}{h_i }
+
\dfrac{D(z_i) + D(z_{i+1})}{h_{i+1} },
& \text{if } i=j\neq n,\; i=j\neq 1 \vspace{2mm} \\

\dfrac{D(z_{n-1}) + D(z_n)}{h_n },
& \text{if } i=j=n \vspace{2mm} \\

-\dfrac{D(z_{i-1}) + D(z_i)}{h_i },
& \text{if } j=i-1,\; i\neq 1 \vspace{2mm} \\

-\dfrac{D(z_i) + D(z_{i+1})}{h_{i+1}},
& \text{if } j=i+1,\; i\neq n \vspace{2mm} \\

0, & \text{otherwise}.
\end{cases}
\qquad
\label{inner_D_prime_prime}
\end{eqnarray}

\subsubsection{f-weighted Inner Products}\label{sec:fip}
In this section, we will approximate the $f$- weighted inner products  $\langle \varphi_i, \varphi_j \rangle_f$ and  $\langle \varphi_i', \varphi_j \rangle_f$.\\
For $\displaystyle \langle \varphi_i, \varphi_j \rangle_f = \int_0^{z_F} f\, \varphi_i \varphi_j \,dz
$, using  MVT \eqref{eq:MVT} and \eqref{inner}
we get:
\begin{itemize}
    \item For \( j = i-1 \), \( i \neq 1 \):
\begin{eqnarray}
        \langle \varphi_i, \varphi_{i-1} \rangle_f &=& \int_{z_{i-1}}^{z_i} f(z) \varphi_i \varphi_{i-1} \, dz \nonumber\\
   & \approx& \frac{f(z_i) + f(z_{i-1})}{2} \cdot \int_{z_{i-1}}^{z_i} \varphi_i \varphi_{i-1} \, dz \;=\;\frac{f(z_i) + f(z_{i-1})}{2} \cdot \frac{z_i - z_{i-1}}{6}.\nonumber
    \end{eqnarray}
    \item For \( j = i+1 \), \( i \neq n \):
\begin{eqnarray}
    \langle \varphi_i, \varphi_{i+1} \rangle_f &=& \int_{z_i}^{z_{i+1}} f(z) \varphi_i \varphi_{i+1} \, dz\nonumber\\
    & \approx& \frac{f(z_{i+1}) + f(z_i)}{2} \cdot \int_{z_i}^{z_{i+1}} \varphi_i \varphi_{i+1} \, dz \approx \frac{f(z_{i+1}) + f(z_i)}{2} \cdot \frac{z_{i+1} - z_i}{6}.\nonumber
    \end{eqnarray}
    \item For \( j = i \neq n,  j=i \neq 1 \):
        \begin{eqnarray}
    \langle \varphi_i, \varphi_i \rangle_f &=& \int_{z_{i-1}}^{z_i} f(z) \varphi_i^2 \, dz + \int_{z_i}^{z_{i+1}} f(z) \varphi_i^2 \, dz\nonumber\\
    &\approx&  \frac{f(z_{i+1}) + f(z_i)}{2}  \cdot \int_{z_i}^{z_{i+1}} \varphi_i^2 \, dz +  \frac{f(z_{i-1}) + f(z_i)}{2}  \cdot \int_{z_{i-1}}^{z_i} \varphi_i^2 \, dz \nonumber\\
   & \approx & \frac{f(z_{i+1}) + f(z_i)}{2}  \cdot \frac{z_{i+1} - z_i}{3} +  \frac{f(z_{i-1}) + f(z_i)}{2}  \cdot \frac{z_i - z_{i-1}}{3}.\nonumber
 \end{eqnarray}
    \item For \( j = i = n \):
    \[
    \langle \varphi_n, \varphi_n \rangle_f = \int_{z_{n-1}}^{z_n} f(z) \varphi_n^2 \, dz \approx \frac{f(z_n) + f(z_{n-1})}{2} \cdot \frac{z_n - z_{n-1}}{3}.
    \]
\item For \( j = i = 1 \):
\[
\langle \varphi_1, \varphi_1 \rangle_f = \int_{z_1}^{z_2} f(z) \varphi_1^2 \, dz \approx \frac{f(z_1) + f(z_2)}{2} \cdot \frac{z_2 - z_1}{3}.
\]
\end{itemize}
So, \begin{eqnarray}
\langle \varphi_i, \varphi_j \rangle_f \approx \dfrac{1}{12}
\begin{cases}
2 h_2\;(f(z_1) + f(z_2)) , & \text{if } i = j = 1, \vspace{2mm} \\
2 h_{i+1}\;(f(z_{i+1}) + f(z_i))  + 2 h_{i}\;(f(z_i) + f(z_{i-1})) , & \text{if } i = j \neq 1, i = j \neq n \vspace{2mm} \\
2h_n\;(f(z_n) + f(z_{n-1})) , & \text{if } i = j = n, \vspace{2mm} \\
h_i\;(f(z_i) + f(z_{i-1})) , & \text{if } j = i-1, i \neq 1 \vspace{2mm} \\
h_{i+1}\;(f(z_{i+1}) + f(z_i)) , & \text{if } j = i+1, i \neq n \vspace{2mm} \\
0, & \text{otherwise}.
\end{cases} \qquad
\label{inner_f}
\end{eqnarray}
\noindent For $\langle \varphi_i', \varphi_j \rangle_f  =  \int_0^{z_F} f\, \varphi_i \varphi_j \,dz
$, using  MVT \eqref{eq:MVT} and \eqref{inner_prime}
we get:

\begin{itemize}
\item For \( j = i - 1 \) and \( i \neq 1 \):
\begin{eqnarray}
\langle \varphi'_i, \varphi_{i-1} \rangle_f 
&=& \int_{z_{i-1}}^{z_i} f(z)\,\varphi'_i \, \varphi_{i-1} \, dz 
\,\approx\, \frac{f(z_{i-1}) + f(z_i)}{2} \cdot \langle \varphi'_i, \varphi_{i-1} \rangle 
\,=\,  \frac{1}{2} \cdot \frac{f(z_{i-1}) + f(z_i)}{2} \nonumber
\end{eqnarray}
\item For \( j = i + 1 \) and \( i \neq n \):
\begin{eqnarray}
\langle \varphi'_i, \varphi_{i+1} \rangle_f 
&=& \int_{z_i}^{z_{i+1}} f(z)\,\varphi'_i \, \varphi_{i+1} \, dz 
\,\approx\, \frac{f(z_i) + f(z_{i+1})}{2} \cdot \langle \varphi'_i, \varphi_{i+1} \rangle
\,=\, -\frac{1}{2}\cdot\frac{f(z_i) + f(z_{i+1})}{2}  \quad \nonumber 
\end{eqnarray}
\item For \( j = i \neq n, j=i \neq 1 \):
\begin{eqnarray}
\langle \varphi'_i, \varphi_i \rangle_f 
&=& \int_{z_{i-1}}^{z_i} f(z)\,\varphi'_i \, \varphi_i \, dz + \int_{z_i}^{z_{i+1}} f(z)\,\varphi'_i \, \varphi_i \, dz \nonumber \\
&\approx& \frac{f(z_{i-1}) + f(z_i)}{2} \cdot \frac{1}{2} \;+\; \frac{f(z_i) + f(z_{i+1})}{2} \cdot \left(-\frac{1}{2}\right) 
\,=\, \frac{f(z_{i-1}) - f(z_{i+1})}{4}\nonumber
\end{eqnarray}
\item For \( i = j = n \):
\begin{eqnarray}
\langle \varphi'_n, \varphi_n \rangle_f 
&=& \int_{z_{n-1}}^{z_n} f(z)\,\varphi'_n \, \varphi_n \, dz 
\,\approx\, \frac{f(z_{n-1}) + f(z_n)}{2} \cdot \langle \varphi'_n, \varphi_n \rangle 
\,=\, \frac{1}{2} \cdot\frac{f(z_{n-1}) + f(z_n)}{2} \nonumber 
\end{eqnarray}

\item For \( j = i = 1 \):
\begin{eqnarray}
\langle \varphi'_1, \varphi_1 \rangle_f
&=& \int_{z_1}^{z_2} f(z)\,\varphi'_1 \, \varphi_1 \, dz 
\,\approx\,\frac{f(z_1) + f(z_2)}{2} \cdot \langle \varphi'_1, \varphi_1 \rangle 
\,=\, -\frac{1}{2} \cdot \frac{f(z_1) + f(z_2)}{2} \nonumber 
\end{eqnarray}
\end{itemize}

Thus,
\begin{eqnarray}
\langle \varphi'_i, \varphi_j \rangle_f \approx \dfrac{1}{4}\;
\begin{cases}
-(f(z_1) + f(z_2)),
& \text{if } i = j = 1, \vspace{1mm} \\
{f(z_{i-1}) - f(z_{i+1})},
& \text{if } i = j \neq 1, i = j \neq n \vspace{2mm} \\
{f(z_{n-1}) + f(z_n)},
& \text{if } i = j = n, \vspace{2mm} \\
{f(z_{i-1} )+ f(z_i)},
& \text{if } j = i-1, i \neq 1 \vspace{2mm} \\
- 
(f(z_i )+ f(z_{i+1})),
& \text{if } j = i+1, i \neq n \vspace{2mm} \\
0, & \text{otherwise}.
\end{cases} \qquad
\label{inner_f_prime}
\end{eqnarray}

\subsection{Vectors}\label{sec:matvec}
Recall that $(\rho ^{\text{atm}}(t+\Delta t) - \rho^{\text{atm}}(t))\langle \varphi_1,\varphi_j\rangle_f$ is equivalent to the vector of length $n-1$,
\begin{eqnarray}
v_1(t) &=& \left( \rho^{\text{atm}}(t + \Delta t) - \rho^{\text{atm}}(t) \right) \langle \varphi_1, \varphi_2 \rangle_f \cdot e_1,\nonumber\\
&\approx&\bigl(\rho^ {\text{atm}}(t + \Delta t)-\rho^ {\text{atm}}(t)\bigr)\,
\frac{f_1+f_2}{12}\,(z_2-z_1)\,e_1\nonumber\\
&=&\bigl(\rho^ {\text{atm}}(t + \Delta t)-\rho ^ {\text{atm}}(t)\bigr)\,
\frac{h_2}{12}\,(f_1+f_2)\,e_1 \label{eq:appv1}
\end{eqnarray}
\noindent based on \eqref{inner_f}, where \( e_1 = \begin{bmatrix} 1 & 0 & \cdots & 0 \end{bmatrix}^T \in \mathbb{R}^{n-1} \) and $h_2 = z_2-z_1 = z_2$.

\noindent Similarly, using \eqref{inner}, \eqref{inner_D_prime}, \eqref{inner_D_prime_prime}, and \eqref{inner_f_prime}, we get:
\begin{eqnarray}
 v_3(t) &=& \rho^ {\text{atm}} (t+\Delta t)\left[\mathcal{G} \ltwoinner{\varphi_1 ,\varphi_2}  + \dfrac{1}{z_F^2} \ltwoinner{D  \varphi_1' ,\varphi_2'}-\dfrac{\mathcal{F}}{z_F} \ltwof{ \varphi_1 ,\varphi_2'}- \dfrac{\mathcal{M}}{z_F}\ltwoinner{D  \varphi_1 ,\varphi_2'} \right]e_1 \nonumber\\
  &\approx & \rho^ {\text{atm}} (t+\Delta t)\left[\mathcal{G} \frac{h_2}{6}  - \dfrac{1}{z_F^2}\frac{D_1+D_2}{2h_2} -\dfrac{\mathcal{F}}{z_F}\frac{f_1+f_2}{4}- \dfrac{\mathcal{M}}{z_F}\frac{D_1+D_2}{4} \right]e_1 \label{eq:appv3}
 \end{eqnarray}
 For uniform meshing, $h_2$ is replaced by $h$ in \eqref{eq:appv1} and \eqref{eq:appv3}.
\subsection{Matrices}\label{sec:matvec2}
\noindent The entries of \( M_f \) are given by the weighted inner products for $i=1,2,\cdots , n-1$:
$$
(M_f)_{i,j} = \langle \varphi_{i+1}, \varphi_{j+1} \rangle_f = \int_0^{z_F} \varphi_{i+1}(z) \varphi_{j+1}(z) f(z) \, dz.
 $$
\noindent Using \eqref{inner_f} ,
the entries are approximated as follows:
\begin{eqnarray*}
(M_f)_{ij} \approx
\begin{cases}
\dfrac{f_{i+2}+f_{i+1}}{2}\,\dfrac{h_{i+2}}{3}
+\dfrac{f_{i+1}+f_i}{2}\,\dfrac{h_{i+1}}{3},
& \text{if } i=j,\; i=1,\ldots,n-2, \vspace{2mm} \\

\dfrac{f_n+f_{n-1}}{2}\,\dfrac{h_n}{3},
& \text{if } i=j=n-1, \vspace{2mm} \\

\dfrac{f_{i+1}+f_i}{2}\,\dfrac{h_{i+1}}{6},
& \text{if } j=i-1,\; i=2,\ldots,n-1, \vspace{2mm} \\

\dfrac{f_{i+2}+f_{i+1}}{2}\,\dfrac{h_{i+2}}{6},
& \text{if } j=i+1,\; i=1,\ldots,n-2, \vspace{2mm} \\

0,
& \text{otherwise}.
\end{cases}
\end{eqnarray*}
where for a nonuniform meshing\\

\resizebox{\textwidth}{!}{
\normalsize
$
M_f \approx
\begin{bmatrix}
\displaystyle
\frac{f_3+f_2}{2}\, \frac{h_3}{3} + \frac{f_2+f_1}{2}\, \frac{h_2}{3} & \displaystyle \frac{f_3+f_2}{2}\, \frac{h_3}{6} & 0 & \cdots & 0 \\[6pt]

\displaystyle
\frac{f_3+f_2}{2}\, \frac{h_3}{6} & \displaystyle \frac{f_4+f_3}{2}\, \frac{h_4}{3} + \frac{f_3+f_2}{2}\, \frac{h_3}{3} & \displaystyle \frac{f_4+f_3}{2}\, \frac{h_4}{6} & \cdots & 0 \\[6pt]

0 & \displaystyle \frac{f_4+f_3}{2}\, \frac{h_4}{6} & \ddots & \ddots & \vdots \\[6pt]

\vdots & \vdots & \ddots & \displaystyle \frac{f_{n-1}+f_{n-2}}{2}\, \frac{h_{n-1}}{3} + \frac{f_{n-2}+f_{n-3}}{2}\, \frac{h_{n-2}}{3} & \displaystyle \frac{f_n+f_{n-1}}{2}\, \frac{h_n}{6} \\[6pt]

0 & 0 & \cdots & \displaystyle \frac{f_n+f_{n-1}}{2}\, \frac{h_n}{6} & \displaystyle \dfrac{f_n+f_{n-1}}{2}\, \dfrac{h_n}{3}
\end{bmatrix}
$
}\\

\noindent and for a uniform mesh\\

\resizebox{\textwidth}{!}{
\scriptsize
$
M_f \approx\dfrac{h}{12}
\begin{bmatrix}
\displaystyle  2(f_1 + 2 f_2 + f_3) & \displaystyle  (f_2 + f_3) & 0 & \cdots & 0 \\[2pt]

\displaystyle  (f_2 + f_3) & \displaystyle  2(f_2 + 2 f_3 + f_4) & \displaystyle  (f_3 + f_4) & \cdots & 0 \\[2pt]

0 & \displaystyle  (f_3 + f_4) & \ddots & \ddots & \vdots \\[2pt]

\vdots & \vdots & \ddots & \displaystyle  2(f_{n-2} + 2 f_{n-1} + f_n) & \displaystyle  (f_{n-1} + f_n) \\[2pt]

0 & 0 & \cdots & \displaystyle  (f_{n-1} + f_n) & \displaystyle  2(f_{n-1} + f_n)
\end{bmatrix}
$}\\\\

\noindent Now, for \( K_f \), its entries are given by
$$
(K_f)_{i,j} := \mathcal{F}\,\langle \varphi_{i+1}',\,\varphi_{j+1} \rangle_f, \quad i,j = 1,\ldots,n-1,
$$
 which are approximated using \eqref{inner_f_prime}, to get
\begin{eqnarray*}
(K_f)_{ij} \approx \dfrac{\mathcal{F}}{4} \;
\begin{cases}
{f_i - f_{i+2}}, 
& \text{if } i=j, \; i=1,\ldots,n-2, \vspace{2mm} \\
{f_{n-1} + f_n}, 
& \text{if } i=j=n-1, \vspace{2mm} \\
{f_i + f_{i+1}}, 
& \text{if } j=i-1, \; i=2,\ldots,n-1, \vspace{2mm} \\
-(f_{i+1} + f_{i+2}), 
& \text{if } j=i+1, \; i=1,\ldots,n-2, \vspace{2mm} \\

0, & \text{otherwise}.
\end{cases}
\end{eqnarray*}

\noindent Thus,\\
\resizebox{\textwidth}{!}{
$
K_f \approx \dfrac{\mathcal{F}}{4}
\begin{bmatrix}
f_1 - f_3 & -(f_2 + f_3) & 0 & 0 & \cdots & 0 \\
f_2 + f_3 & f_2 - f_4 & -(f_3 + f_4) & 0 & \cdots & 0 \\
0 & f_3 + f_4 & f_3 - f_5 & -(f_4 + f_5) & \cdots & 0 \\
0 & 0 & f_4 + f_5 & f_4 - f_6 & \cdots & 0 \\
\vdots & \vdots & \vdots & \vdots & \ddots & -(f_{n-1} + f_{n}) \\
0 & 0 & 0 & 0 & f_{n-1} + f_n & f_{n-1} + f_n
\end{bmatrix}
$}\\
\noindent Similarly, we approximate the matrices $A(D)$ and $S(D)$ whose entries are given by $
A_{i,j} := \langle D\,\varphi_{i+1}',\,\varphi_{j+1} \rangle, 
$  and 
$
S_{i,j} := \langle D \varphi_{i+1}', \varphi_{j+1}' \rangle, 
$
respectively for $ i,j = 1,\ldots,n-1,$.\\ 

\noindent For $A(D)$, we use \eqref{inner_D_prime} to get
\[
A_{i,j} \approx
\begin{cases}
\dfrac{D_i - D_{i+2}}{4}, 
& \text{if } i=j, \; i=1,\ldots,n-2, \vspace{1mm} \\
\dfrac{D_{\,n-1} + D_n}{4}, 
& \text{if } i=j=n-1, \vspace{1mm} \\
\dfrac{D_i + D_{i+1}}{4}, 
& \text{if } j=i-1, \; i=2,\ldots,n-1, \vspace{1mm} \\
-\dfrac{D_{i+1} + D_{i+2}}{4}, 
& \text{if } j=i+1, \; i=1,\ldots,n-2, \vspace{1mm} \\
0, & \text{otherwise},
\end{cases}
\]
where
\[
\resizebox{\textwidth}{!}{$
A(D) \approx \dfrac{1}{4}
\begin{bmatrix}
D_1 - D_3 & -(D_2 + D_3) & 0 & 0 & \cdots & 0 \\[1mm]
D_2 + D_3 & D_2 - D_4 & -(D_3 + D_4) & 0 & \cdots & 0 \\[1mm]
0 & D_3 + D_4 & D_3 - D_5 & -(D_4 + D_5) & \cdots & 0 \\[1mm]
0 & 0 & D_4 + D_5 & D_4 - D_6 & \cdots & 0 \\[1mm]
\vdots & \vdots & \vdots & \vdots & \ddots & -(D_{\,n-1}+D_{\,n}) \\[1mm]
0 & 0 & 0 & 0 & D_{\,n-1}+D_{\,n} & D_{\,n-1}+D_{\,n}
\end{bmatrix}.
$}
\]
For $S(D)$, we use \eqref{inner_D_prime_prime}
to get
\[
S_{ij} \approx
\begin{cases}
\dfrac{D_i + D_{i+1}}{2\,h_{i+1}}
\;+\;
\dfrac{D_{i+1} + D_{i+2}}{2\,h_{i+2}},
& \text{if } i=j,\; i=1,\ldots,n-2, \vspace{2mm} \\

\dfrac{D_{n-1} + D_n}{2\,h_n},
& \text{if } i=j=n-1, \vspace{2mm} \\

-\dfrac{D_i + D_{i+1}}{2\,h_{i+1}},
& \text{if } j=i-1,\; i=2,\ldots,n-1, \vspace{2mm} \\

-\dfrac{D_{i+1} + D_{i+2}}{2\,h_{i+2}},
& \text{if } j=i+1,\; i=1,\ldots,n-2, \vspace{2mm} \\

0, & \text{otherwise}.
\end{cases}
\]

\noindent \noindent where for a nonuniform mesh\\
\[
S(D) \approx
\begin{bmatrix}
\dfrac{D_1}{2h_2} + \dfrac{D_2}{2h_3} & -\dfrac{D_2}{2h_3} & 0 & \cdots & 0 \\
-\dfrac{D_2}{2h_3} & \dfrac{D_2}{2h_3} + \dfrac{D_3}{2h_4} & -\dfrac{D_3}{2h_4} & \cdots & 0 \\
0 & \ddots & \ddots & \ddots & 0 \\
\vdots & \cdots & -\dfrac{D_{n-2}}{2h_{n-1}} & \dfrac{D_{n-2}}{2h_{n-1}} + \dfrac{D_{n-1}}{2h_n} & -\dfrac{D_{n-1}}{2h_n} \\
0 & \cdots & 0 & -\dfrac{D_{n-1}}{2h_n} & \dfrac{D_{n-1}}{2h_n}
\end{bmatrix}
\]

\[
\qquad
+
\begin{bmatrix}
\dfrac{D_2}{2h_2} + \dfrac{D_3}{2h_3} & -\dfrac{D_3}{2h_3} & 0 & \cdots & 0 \\
-\dfrac{D_3}{2h_3} & \dfrac{D_3}{2h_3} + \dfrac{D_4}{2h_4} & -\dfrac{D_4}{2h_4} & \cdots & 0 \\
0 & \ddots & \ddots & \ddots & 0 \\
\vdots & \cdots & -\dfrac{D_{n-1}}{2h_{n-1}} & \dfrac{D_{n-1}}{2h_{n-1}} + \dfrac{D_n}{2h_n} & -\dfrac{D_n}{2h_n} \\
0 & \cdots & 0 & -\dfrac{D_n}{2h_n} & \dfrac{D_n}{2h_n}
\end{bmatrix}_{(n-1)\times(n-1)}
\]
\noindent and for a uniform mesh\\
\begin{equation*}
\resizebox{0.95\textwidth}{!}{$
S(D) \approx \dfrac{1}{2h}
\begin{bmatrix}
D_1 + 2D_2 + D_3 & -(D_2 + D_3) & 0 & \cdots & 0 \\
-(D_2 + D_3) & D_2 + 2D_3 + D_4 & -(D_3 + D_4) &  & \vdots \\
0 & \ddots & \ddots & \ddots & 0 \\
\vdots &  & -(D_{n-2} + D_{n-1}) & D_{n-2} + 2D_{n-1} + D_n & -(D_{n-1} + D_n) \\
0 & \cdots & 0 & -(D_{n-1} + D_n) & D_{n-1} + D_n
\end{bmatrix}
$}
\end{equation*}

\noindent Finally for the mass matrix $M$, given by $M_{i,j} := \langle \varphi_{i+1}, \varphi_{j+1} \rangle,$  we use \eqref{inner} to get
\begin{eqnarray*}
M_{ij}
=
\begin{cases}
\dfrac{h_{i+1}+h_{i+2}}{3},
& \text{if } i=j,\; i=1,\ldots,n-2, \vspace{2mm} \\

\dfrac{h_n}{3},
& \text{if } i=j=n-1, \vspace{2mm} \\

\dfrac{h_{i+1}}{6},
& \text{if } j=i-1,\; i=2,\ldots,n-1, \vspace{2mm} \\

\dfrac{h_{i+2}}{6},
& \text{if } j=i+1,\; i=1,\ldots,n-2, \vspace{2mm} \\

0, & \text{otherwise},
\end{cases}
\end{eqnarray*}
\noindent where for a nonuniform mesh
\begin{eqnarray*}
M = \frac{1}{6}
\begin{bmatrix}
2(h_2 + h_3) & h_3 & 0 & 0 & \cdots & 0 \\
h_3 & 2(h_3 + h_4) & h_4 & 0 & \cdots & 0 \\
0 & h_4 & 2(h_4 + h_5) & h_5 & \cdots & 0 \\
0 & 0 & h_5 & 2(h_5 + h_6) & \cdots & 0 \\
\vdots & \vdots & \vdots & \vdots & \ddots & h_n \\
0 & 0 & 0 & 0 & h_n & 2h_n
\end{bmatrix}_{(n-1)\times(n-1)}.
\end{eqnarray*}
and for a uniform mesh\\
\begin{eqnarray*}
M = \frac{h}{6}
\begin{bmatrix}
4 & 1 & 0 & 0 & \cdots & 0 \\
1 & 4 & 1 & 0 & \cdots & 0 \\
0 & 1 & 4 & 1 & \cdots & 0 \\
0 & 0 & 1 & 4 & \cdots & 0 \\
\vdots & \vdots & \vdots & \vdots & \ddots & 1 \\
0 & 0 & 0 & 0 & 1 & 2
\end{bmatrix}_{(n-1)\times(n-1)}.
\end{eqnarray*}

\end{appendix}
\end{document}